\documentclass[12pt, twoside, reqno]{amsart}

\usepackage{amsmath}
\usepackage{amsfonts}
\usepackage{amssymb}
\usepackage{amsthm}
\usepackage{hyperref}
\usepackage{graphicx}
\usepackage{cite}
\usepackage{caption}
\usepackage[all]{xy}

\usepackage[top=1in, bottom=1.25in, left=1.25in, right=1.25in, headsep=0.2in]{geometry}

\usepackage{fancyhdr}

\DeclareMathOperator{\adj}{adj}
\DeclareMathOperator{\im}{Im}
\DeclareMathOperator{\re}{Re}
\DeclareMathOperator{\supp}{supp}
\DeclareMathOperator{\spann}{span}

\theoremstyle{plain}
\newtheorem{theorem}{Theorem}[section]
\newtheorem{definition}[theorem]{Definition}
\newtheorem{lemma}[theorem]{Lemma}
\newtheorem{rem}[theorem]{Remark}
\newtheorem{prop}[theorem]{Proposition}

\newtheorem{condition}[theorem]{Condition}
\newtheorem{ex}[theorem]{Example}
\newtheorem{corollary}[theorem]{Corollary}

\numberwithin{equation}{section}

\newcommand{\Lapl}{\mathcal{L}}

\begin{document}

\title{Calder\'{o}n problem for Yang-Mills connections}
\author[M. Ceki\'{c}]{Mihajlo Ceki\'{c}}
\address{Max-Planck Institute for Mathematics, Vivatsgasse 7, 53111, Bonn, Germany}
\email{m.cekic@mpim-bonn.mpg.de}

\maketitle

\begin{abstract}
We consider the problem of identifying a unitary Yang-Mills connection $\nabla$ on a Hermitian vector bundle from the Dirichlet-to-Neumann (DN) map of the connection Laplacian $\nabla^*\nabla$ over compact Riemannian manifolds with boundary. We establish uniqueness of the connection up to a gauge equivalence in the case of trivial line bundles in the smooth category and for the higher rank case in the analytic category, by using geometric analysis methods and essentially only one measurement. 

Moreover, by using a Runge-type approximation argument along curves to recover holonomy, we are able to uniquely determine both the bundle structure and the connection, but at the cost of having more measurements. Also, we prove that the DN map is an elliptic pseudodifferential operator of order one on the restriction of the vector bundle to the boundary, whose full symbol determines the complete Taylor series of an arbitrary connection, metric and an associated potential at the boundary.
\end{abstract}

\section{Introduction}

In this paper, we consider the Calder\'{o}n inverse problem for a special type of connections, called the Yang-Mills connections. Given a Hermitian vector bundle $E$ of rank $m$ over a compact Riemannian manifold $(M, g)$ with non-empty boundary and a unitary connection $A$ ($\nabla$) on $E$, one may consider the connection Laplacian denoted by $d_A^*d_A$ ($\nabla^*\nabla$), where $d_A^*$ ($\nabla^*$) denotes the formal adjoint of $d_A$ ($\nabla$) with respect to the Hermitian and Riemannian structures. Sometimes this operator is called the magnetic Laplacian because it is used to represent the magnetic Schr\"{o}dinger equation, where $A$ corresponds to the magnetic potential.

Given this, we may define the associated Dirichlet-to-Neumann (DN map in short) $\Lambda_A:C^{\infty}(\partial M; E|_{\partial M}) \to C^{\infty}(\partial M; E|_{\partial M})$\footnote{By $C^{\infty}(M; E)$ we denote the space of smooth sections of $E$ over $M$.} by solving the Dirichlet problem:
\begin{align}\label{DNmap}
d_A^*d_A(u) = 0, \quad u|_{\partial M} = f
\end{align}
and setting $\Lambda_A(f) = d_A(u)(\nu)$, where $\nu$ is the outwards pointing normal at the boundary. The problem can then be posed as asking whether the map $A \mapsto \Lambda_A$ is injective modulo the natural obstruction, or in other words whether $\Lambda_A = \Lambda_B$ implies the existence of a gauge automorphism $F:E \to E$ with $F^{*}(A) = B$ and $F|_{\partial M} = Id$.

This problem was considered in \cite{cek1, Eskin, MagU, LCW}; for a survey of the Calder\'{o}n problem for metrics, see \cite{survey}. In this paper, we take two approaches to uniqueness: one is via geometric analysis and the other by constructing special gauges along curves via the Runge approximation property of elliptic equations. As far as we know, this paper is the first one that considers the connection problem and does not rely on the Complex Geometric Optics solutions (see any of \cite{cek1, Eskin, MagU, LCW}), but on unique continuation principles and geometric analysis of the zero set of a solution to an elliptic equation.

The Yang-Mills connections generalise flat connections and are important in physics and geometry. They satisfy the following equation:
\begin{align*}
D_A^*F_A = 0 
\end{align*}
where $D_A = d_A^{\text{End}}$ is the induced connection on the endomorphism bundle $\text{End} E$ and $F_A$ is the curvature of $A$ (see the preliminaries for more details).

Firstly, we prove that the DN map $\Lambda_A$ is an elliptic pseudodifferential operator of order $1$ on the restriction of the vector bundle to the boundary and deduce that its full symbol determines the full Taylor series of the connection, metric and a potential at the boundary. This was first proved in the case of a Riemannian metric by Lee and Uhlmann \cite{leeuhlmann} and later considered in the $m = 1$ case with a connection in \cite{LCW}. In this paper, we generalise this approach to the case of systems and prove the analogous result.

\subsection{Motivation.} Let us explain some motivation for considering this problem. Partly, the idea came from the analogy between Einstein metrics in Riemannian geometry and Yang-Mills connections on Hermitian vector bundles. Also, Guillarmou and S\'{a} Barreto in \cite{guill} prove the recovery of two Einstein manifolds from the DN map for metrics. The method of their proof relies on a reconstruction near the boundary, where in special harmonic coordinates Einstein equations become quasi-linear elliptic (the metric is thus also analytic in such coordinates). Hence, by combining the boundary determination result and a unique continuation result for elliptic systems they prove one can identify the two metrics in a neighbourhood of the boundary. Moreover, by exploiting this analytic structure they observe that the method of Lassas and Uhlmann \cite{calderon_analytic} who prove the analytic Calder\'{o}n problem for metrics, may be used to extend this local isometry to the whole of the manifold.\footnote{This works by embedding the two manifolds in a suitable Sobolev space using Green's functions of the metric Laplacians and showing the appropriate composition is an isometry.}

In our case, the conventionally analogous concept to harmonic coordinates to consider would be the \emph{Coulomb gauge} \cite{uhlenbeck} which transforms the connection to a form where $d^*(A) = 0$, so that the Yang-Mills equations become an elliptic system with principal diagonal part. \emph{However}, this gauge does not tie well with the DN map, so in Lemma \ref{gauge_constr} we construct an \emph{analogue of the harmonic gauge for connections}. In this gauge, we may use a similar unique continuation property (UCP in short) result to yield the equivalence of connections close to the boundary. However, for going further into the interior we designed new methods.


\subsection{Uniqueness via geometric analysis.} We believe this approach to be entirely new. Here is one of the main theorems of the paper.

\begin{theorem}[Global result]\label{linebundle}
Assume $\dim M \geq 2$, let $E = M \times \mathbb{C}$ be a Hermitian line bundle with standard metric and $\emptyset \neq \Gamma \subset \partial M$ an open, non-empty subset of the boundary. Let $A$ and $B$ be two unitary Yang-Mills connections on $E$. If $\Lambda_A(f)|_{\Gamma} = \Lambda_B(f)|_{\Gamma}$ for all $f \in C_0^\infty(\Gamma; E|_\Gamma)$, then there exists a gauge automorphism (unitary) $h$ with $h|_{\Gamma} = Id$ such that $h^*(A) = B$ on the whole of $M$.
\end{theorem}


We now explain this geometric analysis type method in more detail. Our gauge $F$ from Lemma \ref{gauge_constr} ($m \times m$ matrix function on $M$) satisfies the equation $d_A^*d_A F = 0$ and so we cannot guarantee that it is non-singular globally. We show that the zero set of the determinant of $F$ is suitably small in the smooth case when $m = 1$ and in the analytic case for arbitrary $m$ -- it is covered by countably many submanifolds of codimension one, or in the language of geometric analysis it is $(n-1)$-$C^\infty$-rectifiable. Since (the complement of) this singular set can be topologically non-trivial (see Figure \ref{zero_set}), we end up with barriers consisting of singular points of $F$ that prevent us to use the UCP and go inside the manifold. This is addressed by looking at the sufficiently nice points of the barriers and locally near these points, using a degenerate form of UCP (in the smooth case) or a suitable form of analytic continuation (in the analytic case) to extend an appropriate gauge equivalence between the two given connections beyond the barriers; we name this procedure as ``drilling". Since we show there is a dense set of such nice points, we may perform the drilling to extend our gauges globally.

Here is what we prove in the analytic case, for arbitrary $m$:

\begin{theorem}\label{analytic_m>1}
Let $(M, g)$ be an analytic Riemannian manifold with $\dim M \geq 2$ and let $\Gamma$ be as in Theorem \ref{linebundle}.\footnote{The metric $g$ is only assumed to be analytic in the interior of $M$ and smooth up to the boundary.} If $E = M \times \mathbb{C}^{m}$ is a Hermitian vector bundle with the standard structure and if $A$ and $B$ are two unitary Yang-Mills connections on $E$, then $\Lambda_A(f)|_\Gamma = \Lambda_B(f)_\Gamma$ for all $f \in C_0^\infty(\Gamma; E|_\Gamma)$ if and only if there exists a gauge automorphism $H$ of $E$, with $H|_{\Gamma} = Id$, such that $H^*(A) = B$.
\end{theorem}

We briefly remark that the proof the above theorem also relies on using the Coulomb gauge locally, since in this gauge the connection is analytic, so $\det F$ satisfies the SUCP (see Lemma \ref{recYM-analytic}); we additionally apply this gauge in the drilling procedure.

The main difficulty in proving uniqueness via the geometric analysis method for the smooth, higher rank ($m > 1$) case is that the strong unique continuation property (SUCP) for the determinant $\det{F}$ of a solution to $d_A^*d_A F = 0$ might not hold -- see Remark \ref{rem:smoothm>1} for more details. Indeed, in the subsequent work \cite{cek2} we treat this question in more detail and prove a positive answer for $n = 2$ and also provide some counterexamples.

\subsection{Uniqueness via Runge approximation.}
Next, we outline our second approach to uniqueness by using Runge-type approximation for elliptic equations to recover holonomy, which we use to prove the stronger statement of uniqueness for arbitrary bundles. In general, Runge approximation is known to be applicable to inverse problems (see e.g. \cite{conformalcalderon, BM13, RS17}).

\begin{theorem}\label{thm:mainthm'}
    Let $(M, g)$ be a smooth compact Riemannian manifold with boundary of dimension $\dim M \geq 2$, $\Gamma \subset \partial M$ a non-empty open set, $E$ and $E'$ Hermitian vector bundles over $M$ such that we have the identification $E|_{\Gamma} = E'|_{\Gamma}$. Let $A$ and $B$ be two smooth unitary Yang-Mills connections on $E$ and $E'$ respectively, such that $\Lambda_A(f)|_{\Gamma} = \Lambda_B(f)|_{\Gamma}$ for all $f \in C_0^\infty(\Gamma; E)$. Then there exists a unitary bundle isomorphism $H:E' \to E$ with $H|_{\Gamma} = Id$, such that $H^*A = B$.
\end{theorem}


This theorem clearly generalises Theorems \ref{linebundle} and \ref{analytic_m>1}. Note in particular that we are able to uniquely determine the topology of the bundles and the Hermitian structures. 

Let us point out the main differences between the two approaches. Although Theorem \ref{thm:mainthm'} is stronger than the first two theorems, the advantage of the former approach is in its method of proof. More precisely, the geometric analysis technique is minimal with respect to the necessary data -- essentially, we only need one arbitrary measurement to uniquely identify the connections -- see Remark \ref{rem:minimalcondition} for more details. Also, this method is entirely new, so it gives hope that it can be generalised to different settings, such as the metric Calder\'on problem. 

On the other hand, in the Runge-type density approach we need many measurements that concentrate in a suitable sense on closed loops (see Lemma \ref{lemma:runge}). In the context of the Calder\'on problem for connections, this method is also new and gives hope to be generalised to other settings.




\subsection{Organisation of the paper.} The paper is organised as follows: in the next section, we recall some formulas from differential geometry and a make a few observations about choosing appropriate gauges. In the third section we prove that $\Lambda_A$ is a pseudodifferential operator of order $1$ for systems and prove that its full symbol determines the full jet of $A$ at the boundary. Furthermore, in section four we consider the smooth case and prove the global result for $m = 1$, Theorem \ref{linebundle}. In the same section, we construct the new gauge and deduce the UCP result we need. In section five we consider the $m > 1$ case for analytic metrics, Theorem \ref{analytic_m>1}, by adapting the proof of the line bundle case and exploiting real-analyticity. Finally, in section five we apply the Runge-type approximation property and prove Theorem \ref{thm:mainthm'}.

This paper also has two appendices: in Appendix \ref{app:A} we recall some well-posedness condition for the heat equation and prove a few elementary statements about extending functions smoothly over small sets. Next, in Appendix \ref{app:B} we lay out the technical results needed to prove the Runge type approximation result we need -- this requires some well-posedness for Dirichlet problem in negative Sobolev spaces and a duality argument; the aim is to prove that for $\Lapl_A = d_A^*d_A$, we can build $m$ smooth solutions that span the bundle over a given curve.\\

\noindent \textbf{Acknowledgements.} The author would like to express his gratitude for the support of his supervisor, Gabriel Paternain. Furthermore, he would also like to acknowledge the help of Mikko Salo, Rafe Mazzeo, Colin Guillarmou and Gunther Uhlmann, especially in clarifying the unique continuation results used in the paper. The author thanks Guillaume Bal and Francois Monard for telling him about Runge type approximation results which inspired the last section of the paper. 

He is grateful to the Trinity College for financial support, where the most of this research took place and to Max-Planck Institute for Mathematics in Bonn, where a part of this research took place. 

\section{Preliminaries}

\subsection{Yang-Mills connections.}

As mentioned previously, Yang-Mills (YM) connections are very important in physics and geometry. They satisfy the so called \emph{Yang-Mills equations}, which are considered as a generalisation of Maxwell's equations in electromagnetism and which provide a framework to write the latter equations in a coordinate-free way (see e.g. \cite{ab} or \cite{donaldson} for a geometric overview and definitions). The Yang-Mills connections are critical points of the functional:
\begin{align*}
F_{YM}(A) = \int_M |F_A|^2 d\omega_g
\end{align*}
Here $F_A = dA + A\wedge A$ is the curvature $2$-form with values in the endomorphism bundle of $E$ determined by the map $d_A^2s = F_A \wedge s$ on sections $s \in C^{\infty}(M; E)$ and $\omega_g$ is the volume form. It can then be shown by considering variations of this functional, that the equivalent conditions for $A$ being its critical point are (the Euler-Lagrange equations):
\begin{align}\label{YM0}
(D_A)^* F_A = 0 \quad \text{ and } \quad D_A F_A = 0
\end{align}
where $D_A = d^{\text{End}}_A$ is the induced connection on the endomorphism bundle, given locally by $D_AS = dS + [A, S]$ or equivalently by $D_AS = [d_A, S]$, where $[\cdot, \cdot]$ denotes the commutator. The second equation in \eqref{YM0} is actually redundant, since it is the Bianchi identity.

Yang-Mills connections clearly generalise \emph{flat} connections, for which the curvature vanishes, i.e. $F_A = 0$. 

They have been a point of unification between pure mathematics and theoretical physics, but moreover have brought a few areas of pure mathematics together, such as e.g. PDE theory and vector bundles over complex projective spaces, or algebraic geometry.

\begin{ex}[Yang-Mills connections over Riemann surfaces]\rm
We give an idea of the size of the set of YM connections in the simplest non-trivial example of Riemann surfaces. First recall that connections on bundles modulo gauges are classified by their holonomy representation on the so called loop group modulo conjugation (see Kobayashi and Nomizu \cite{KN63}). In the setting of flat connections, this correspondence simplifies significantly for a Riemann surface $\Sigma$:
\begin{align*}
\big\{\rho: \pi_1(\Sigma) \to U(m)\big\}/\text{conj.} \quad \longleftrightarrow \quad  \{\text{unitary flat bundles of rank } m\}
\end{align*} 
since homotopic loops have the same holonomy. The direct map (going left to right) here is the one taking a representation $\rho$ and defining an associated flat bundle via $\widetilde{\Sigma} \times_\rho \mathbb{C}^m$, where $\widetilde{\Sigma}$ is the universal cover of $\Sigma$ and $\times_\rho$ means we identified the two by the diagonal action. Somewhat surprisingly, we may still obtain a correspondence in the case of YM connections, where $\pi_1(\Sigma)$ is replaced by a certain central extension $\widehat{\pi}_1(\Sigma)$ (see \cite{ab} for more details). This has an analogous geometric interpretation: the difference to the flat case is that we now identify homotopic only if they enclose the same area. In particular, for the sphere $S^2$ this simplifies, so that we have $\widehat{\pi}_1(S^2) = S^1$.
\end{ex}

\subsection{Local expressions for $d_A^*$ and inner products} In general, we use the notation $d_A^*$ to denote the formal adjoint acting on vector valued $p$-forms; if $A$ is unitary, then $d_A^* = (-1)^{(p-1)n + 1}\star d_{A} \star$, where $\star$ is the Hodge star acting $\mathbb{C}$-linearly on differential forms with values in $E$ as $\star(\omega \otimes s) = (\star \omega) \otimes s$, $\omega$ is a differential form and $s$ is a section of $E$.

For the record, we will write down the explicit formula in local coordinates for the inner product on the differential forms with values in $E$. If two $p$-differential forms with values in $E$ are given locally by $\alpha = \sum \alpha_I dx^I$ and $\beta = \sum \beta_J dx^J$ then:\footnote{The factor of $\frac{1}{p!}$ comes from the fact that we want to have $\langle{dx^{i_1} \wedge \dotso \wedge dx^{i_p}, dx^{j_1} \wedge \dotso \wedge dx^{j_p}}\rangle = \det{(g^{i_k j_k})}$.}
\begin{align*}
\langle{\alpha, \beta}\rangle_{\Omega^p(E)} = \frac{1}{p!}g^{i_1j_1}\cdots g^{i_pj_p}\langle{\alpha_{i_1\dotso i_p}, \beta
_{j_1\dotso j_p}}\rangle_E
\end{align*}
Here $\langle{\cdot, \cdot}\rangle_E$ is the inner product in $E$ and $g^{ij}$ denotes the inverse matrix of the metric in local coordinates $g_{ij}$. Moreover, we state the following formula for the adjoint $d^* = (-1)^p \star^{-1} d \star = (-1)^{(p-1)n + 1} \star d \star$, acting on $p$-forms:\footnote{We are assuming that the tensor representing the form is alternating, i.e. we get a minus sign after swapping any two indices.}
\begin{align*}
(d^*\alpha)_{\mu_1 \dotso \mu_{p - 1}} = - g_{\mu_1\nu_1}\cdots g_{\mu_{p-1} \nu_{p-1}}\frac{1}{\sqrt{|\det{g}|}} \partial_\nu \big( \sqrt{|\det{g}|} g^{\nu \lambda} g^{\nu_1 \lambda_1} \cdots g^{\nu_{p-1} \lambda_{p-1}} \alpha_{\lambda \lambda_1 \dotso \lambda_{p-1}} \big)
\end{align*}
We can combine this information along with the condition that $\int{\langle
{d_A^* \alpha, \beta}\rangle_E} = \int{\langle{\alpha, d_A \beta}\rangle_E}$ for all $p$-forms $\beta$ and $(p+1)$-forms $\alpha$, compactly supported in the interior. Then we get:
\begin{align}\label{twistedcodiff}
d_A^* \alpha= d^* \alpha - \sum_{i_1 < \dotso  < i_p}g^{\nu \lambda}A_\nu \alpha_{\lambda i_1 \dotso i_p} dx_{i_1}\wedge \dotso \wedge dx_{i_p}
\end{align}
and as a shorthand we may use $(A, \alpha) = \iota_{A^\sharp}\alpha$ for the sum in the above expression. Here $\sharp$ denotes the isomorphism between $TM$ and $T^*M$ given by contracting the metric $g$ with a vector. The following identity is also very useful:
\begin{equation*}
d^*(f \omega) = fd^*(\omega) - \iota_{\nabla f} (\omega)
\end{equation*}



If the connection is not unitary, then the expression $(-1)^{(p-1)n + 1}\star d_{(-A^*)} \star$ gives the formal adjoint in a local trivialisation (we assume the Hermitian structure is trivial) on $p$-forms, where $A^*$ denotes the Hermitian conjugate. However, we still have the formula $d_{A'}^{*'} = F^{-1}d_AF$, where $F$ is a local gauge, $A' = F^*A$ and the adjoint $*'$ is taken w.r.t. the pulled back structure. For all $E$-valued $1$-forms $u$ and unitary $A$:
\begin{align}\label{expansion}
d_A^*d_Au = d^*du + d^*(Au) - (A, du) - (A, Au) 
\end{align}


\subsection{Fixing gauges} 

In many mathematical problems and physical situations there exist certain degrees of freedom called \textit{gauges}. More specifically, in our case a gauge is an automorphism of a vector bundle (preserves its structure); then the gauges act on the affine space of connections on this vector bundle by pullback. Here, we make a few remarks about the possible gauges one could use. 

\begin{ex}[An electromagnetic correspondence]\rm
In physics we use the electromagnetic four-potential to describe the electromagnetic field. This potential can be naturally identified (via musical isomorphism, the inverse of $\sharp$) with a connection $1$-form $A$ on the unitary trivial line bundle over the space-time $\mathbb{R}^4$ in the Minkowski metric, so that the actual electromagnetic field is given by the curvature $F = dA$, which is a tensor consisting of six components; the Maxwell's equations then reduce to $d^*dA = 0$ (see \eqref{YM0}). 

\end{ex}
There are several gauges that have proved to work well in practise, i.e. that fit well into other mathematical formalism in applications. One of them is the \textit{Coulomb gauge}, which for a connection matrix on a vector bundle, locally asks that $d^*A = 0$\footnote{This is equivalent to $\nabla \circ \vec{A} = 0$ in the case of $\mathbb{R}^3$ considered in the previous paragraph.} The existence of such gauges is proved by Uhlenbeck \cite{uhlenbeck} for vector bundles over unit balls (see also \cite{donaldson}) under a smallness condition on the $L^p$ norm of the curvature (for specific values of $p$), which locally on a manifold we can always assume if we shrink the neighbourhood sufficiently and then dilate to the unit ball. Most importantly, in such a gauge the Yang-Mills connections satisfy an elliptic partial differential equation with the principal, second order term equal to $(dd^* + d^*d) \times Id$, which is clearly elliptic.

Another slightly related gauge is the \emph{temporal gauge}, which we will also make use of -- in this gauge, one of the components of the connection vanishes locally (we usually distinguish this variable as ``time"). That is, given a local coordinate system $(x_1, \dotso, x_{n-1}, t) = (x, t)$ defined for $t = 0$ and a connection matrix $A = A_i dx^i + A_t dt$, we may solve:
\begin{align*}
\frac{\partial F}{\partial t}(x, t) + A_t(x, t) F(x, t) = 0 \quad \text{and} \quad F(x, 0) = Id
\end{align*}
parametrically smoothly depending on $x$ (the parallel transport equation). Then by definition near $t = 0$, we have $A' = F^*(A) = F^{-1}dF + F^{-1}AF$ satisfying $A'_t = 0$. In this way we may prove Lemma 6.2 in \cite{cek1}, which we state for convenience, since it will get used frequently throughout the paper:

\begin{lemma}\label{lemma}
Let $A$ and $B$ be two unitary connections on a Hermitian vector bundle $E$ over $M$. Consider the tubular neighbourhood $\partial M \times [0, \epsilon)$ of the boundary for some $\epsilon > 0$ and denote the normal distance coordinate (from $\partial M$) by $t$. Then $B$ is gauge equivalent to a unitary connection $B'$ via an automorphism $F$ of $E$ such that $F|_{\partial M} = Id$ and $(B'-A)(\frac{\partial}{\partial t}) = 0$ in the neighbourhood $\partial M \times [0, \delta)$ of the boundary, for some $\delta > 0$.

In particular, if $E = M \times \mathbb{C}^m$ we have gauges $F$ and $G$ for $A$ and $B$ respectively with $F|_{\partial M} = G|_{\partial M} = Id$, such that $A' = F^*A$ and $B' = G^*B$ satisfy $A'(\frac{\partial}{\partial t}) = B'(\frac{\partial}{\partial t}) = 0$ near the boundary.
\end{lemma}

In the situation of this Yang-Mills problem, we would like to use the gauge given by Lemma \ref{lemma} in combination with Lemma \ref{gauge_constr}, because the latter one is intimately tied with the DN map \eqref{DNmap} and allows us to make use of the information packed in the equality $\Lambda_A = \Lambda_B$ for two connections $A$ and $B$.

\section{Boundary determination for a connection and a matrix potential}

In this section, we prove that if we put the connection in a suitable gauge and ``normalise" the metric appropriately, we may determine the full Taylor series of a connection, metric and matrix potential from the DN map on a vector bundle with $m > 1$. The case of $m = 1$ was already considered in \cite{LCW} (Section 8) and we generalise the result proved there. The approach is based on constructing a factorisation of the operator $d_A^*d_A + Q$ modulo smoothing, from which we deduce that $\Lambda_{g, A, Q}$ is a pseudodifferential operator of order one whose full symbol determines the mentioned Taylor series.

\subsection{PDOs on vector bundles.}
Before going into proofs, let us briefly lay out some of the notation that goes into pseudodifferential operators on vector bundles over manifolds (see \cite{ML89, treves, Shu01} for more details). The local symbol calculus developed for scalar operators carries over to the case of vector bundles, as can be seen from the above references.

So given $X \subset \mathbb{R}^n$ open, $k, l \in \mathbb{N}$ and $m \in \mathbb{R}$, we have the left symbol classes (and more generally, $(x,y)$-dependant symbols) $S^{m}(X; \mathbb{C}^{lk})$ of $l$ by $k$ matrices, whose entries are symbols in $S^{m}(X)$ -- this symbol class yields a map $A: C_{0}^\infty(X, \mathcal{C}^k) \to C^\infty(X, \mathbb{C}^l)$ via the formula 
\begin{align*}
Au(x) = (2\pi)^{-n} \int \int e^{i(x - y) \cdot \xi} a(x, \xi) u(y) dy d\xi
\end{align*}
We say $A$ belongs to the class $\Psi^{m}(X; \mathbb{C}^{lk})$ and is PDO of order $m$; we also say $A$ is \emph{classical} if its symbol is a sum of positive homogeneous symbols. 

Then given a Riemannian manifold $M$ and vector bundles $E$ and $F$ over $M$, we say that a linear map $A: C_{0}^\infty(M; E) \to C^\infty(M; F)$ is a PDO of order $m$ if for all charts and trivialisations of $E$ and $F$ over this chart, the induced map in the local chart is in $\Psi^{m}$. We write $A \in \Psi^{m}(M; E, F)$ for the space of PDOs of order $m$ and define the space of \emph{smoothing} operators $\Psi^{-\infty}(M; E, F) = \cap_{m} \Psi^{m}(M; E, F)$; we will abbreviate $\Psi^{m}(X; E) := \Psi^{m}(X; E, E)$. Such an operator extends by duality to a map $A: \mathcal{E}'(X, E) \to \mathcal{D}'(X, F)$ (the transpose ${}^{t}A$ is defined by taking the transpose of the symbol $a$ and swapping $x$ and $y$).

When $M$ is closed, it is standard that $A \in \Psi^m$ maps the Sobolev space of sections in $H^s$ to $H^{s-m}$.

Care should be taken when considering the composition calculus, since commutation properties of matrices jumps into play. More precisely, we have the following composition formula (see the proof of Theorem 4.3 in \cite{treves}), which computes the symbol $c$ modulo $S^{-\infty}$ of the composition $C = A \circ B$ of two matrix valued pseudodifferential operators $A$ ($k$ by $l$) and $B$ ($l$ by $r$) with symbols $a$ and $b$, respectively:

\begin{align}\label{symbolcomposition}
c(x, \xi) \sim \sum_\alpha \frac{1}{\alpha!} \partial_\xi^\alpha a(x, \xi) D_x^\alpha b(x, \xi)
\end{align}
Finally, we remark that the globally defined principal symbol of a PDO $A \in \Psi^{m} (M; E, F)$ is a well-defined element of the quotient
\[\sigma_{m}(A) \in S^{m}\Big(M; \text{Hom}\big(\pi^*(E), \pi^*(F)\big)\Big)/S^{m-1}\Big(M; \text{Hom}\big(\pi^*(E), \pi^*(F)\big)\Big)\]
where $\pi: T^*M \to M$ denotes the projection, $\pi^*$ is the pullback and Hom is the homomorphism bundle.
 

\begin{rem}\label{rem1}\rm One of the things that fails to hold for matrix pseudodifferential operators and holds for scalar ones, is that commutation decreases degree of the operator by one. However, the following formula still holds if $c$ denotes the symbol of $C = [A, B]$ (commutator bracket) and $a \in S^{m}(X; \mathbb{C}^{l^2})$, $b \in S^{m'}(X; \mathbb{C}^{l^2})$ are the symbols of $A$, $B$, respectively:
\begin{align*}
c(x, \xi) = [a, b](x, \xi) + \frac{h}{i}\{a, b\}(x, \xi) \quad \text{modulo } S^{m + m' - 2}
\end{align*}
where $\{a, b\}(x, \xi) = \sum_{j = 1}^n \big(\frac{\partial a}{\partial \xi_j} \frac{\partial b}{\partial x^j} - \frac{\partial b}{\partial \xi_j} \frac{\partial a}{\partial x^j}\big)$ denotes the matrix valued Poisson bracket.
\end{rem}

\subsection{Boundary determination.}
We are now ready for the main proofs -- we assume that $(M, g)$ is a compact $n$-dimensional manifold with non-empty boundary $N = \partial M$ and $E = M \times \mathbb{C}^m$ a Hermitian vector bundle with a unitary connection $A$ and $Q$ an $m \times m$ matrix whose entries are smooth functions. We will be working in semigeodesic coordinates near $\partial M$ and we denote by $x^n$ the normal coordinate and by $x' = (x^1, x^2, \dotso x^{n-1})$ the local coordinates in $\partial M$. Furthermore, we have in these coordinates that $g = \sum_{\alpha, \beta} g_{\alpha \beta}(x) dx^\alpha dx^\beta + (dx^n)^2$; also, in what follows the summation convention will be used to sum over repeated indices and when using Greek indices $\alpha$ and $\beta$, the summation will always be assumed to go over $1, \dotso, n-1$. We use the notation $D_{x^j} = -i\partial_{x^j} = -i\frac{\partial}{\partial x^j}$ and $|g| = \det{(g_{ij})} = \det{(g_{\alpha \beta})}$. We start by proving an analogue of Lemma 8.6 in \cite{LCW} and Proposition 1.1 in \cite{leeuhlmann}. 

\begin{lemma}\label{factorisationlemma}
Let us assume $A$ satisfies condition \eqref{cond2}. There exists a $\mathbb{C}^{m \times m}$-valued pseudodifferential operator $B(x, D_{x'})$ of order one on $\partial M$, depending smoothly on $x^n \in [0, T]$ for some $T>0$, such that the following factorisation holds:
\begin{align}\label{factorisation}
d_A^*d_A + Q = (D_{x^n} \times Id + i E(x) \times Id - iB(x, D_{x'}))(D_{x^n} \times Id + iB(x, D_{x'}))
\end{align}
modulo smoothing, where $E(x) = -\frac{1}{2}g^{\alpha \beta}(x) \partial_{x^n} g_{\alpha \beta}(x)$.
\end{lemma}
\begin{proof}
First of all, we have that:
\begin{align}\label{expansion'}
(d_A^*d_A + Q)u = \Delta_g (u) - 2g^{jk}A_j\frac{\partial u}{\partial x^k} + (d^*A)u -  g^{jk}A_jA_ku + Qu
\end{align}
where $A = A_i dx^i$. Furthermore, we have 
\begin{align*}
\Delta_g = D_{x^n}^2 + iED_{x^n} + Q_1 + Q_2
\end{align*}
where 
\begin{multline*}
Q_1(x, D_{x'}) = -i\big(\frac{1}{2}g^{\alpha \beta}(x) \partial_{x^\alpha} \log{|g|}(x) + \partial_{x^\alpha} g^{\alpha \beta} (x)\big)D_{x^\beta} \quad \text{ and } \quad Q_2(x, D_{x'})\\ = g^{\alpha \beta} D_{x^\alpha} D_{x^\beta}
\end{multline*}
We denote the symbols of $Q_1$ and $Q_2$ by $q_1$ and $q_2$ respectively and define $G = (d^*A) - g^{\alpha \beta}A_\alpha A_\beta + Q$. Thus by using \eqref{expansion'}, we can rewrite \eqref{factorisation} as
\begin{align*}
B^2  - EB + i[D_{x^n} \times Id, B] = Q_1 \times Id + Q_2 \times Id - 2g^{\alpha \beta} A_\alpha \partial_{x^\beta} + G
\end{align*}
modulo smoothing. Moreover, by taking symbols we obtain:
\begin{align}\label{symbol1}
\sum_{\alpha \geq 0} \frac{1}{\alpha !} \partial_{\xi'}^\alpha b D_{x'}^\alpha b - Eb + \partial_{x^n}b - q_1 \times Id - q_2 \times Id + 2ig^{\alpha \beta}A_\alpha \xi_\beta - G = 0
\end{align}
modulo $S^{-\infty}$, where $b$ is the symbol of $B$ and we have used \eqref{symbolcomposition} and Remark \ref{rem1}. Let us put $b(x, \xi') = \sum_{j \leq 1} b_j(x, \xi')$, where $b_j$ is homogeneous of order $j$ in $\xi'$. We may then determine $b_j$ inductively, starting from degree two in \eqref{symbol1}:
\begin{align}\label{b1}
(b_1)^2 = q_2
\end{align}
so we may set $b_1 = -\sqrt{q_2} \times Id$ (this sign will be important later) and $q_2 = g^{\alpha \beta} \xi_\alpha \xi_\beta$. Next, we have:
\begin{align}\label{b0}
b_0 &= \frac{1}{2\sqrt{q_2}}\Big(\partial_{x^n}b_1 - Eb_1 - q_1 \times Id + 2ig^{\alpha \beta} A_\alpha \xi_\beta + \nabla_{\xi'} b_1 \cdot \nabla_{x'}b_1\Big)\\\label{b-1}
b_{-1} &= \frac{1}{2\sqrt{q_2}}\Big(\partial_{x^n} b_0 - Eb_0 - G + \sum_{0 \leq j, k \leq 1, \text{ } j + k = |K|} \frac{\partial_{\xi'}^K b_j D_{x'}^K b_{|K| - j}}{K!}\Big)\\
b_{m-1} &= \frac{1}{2\sqrt{q_2}}\Big(\partial_{x^{n}}b_m - Eb_m + \sum_{m \leq j, k \leq 1, \text{ } j + k = |K| + m} \frac{\partial_{\xi'}^K b_j D_{x'}^K b_{k}}{K!}\Big)\label{bm}
\end{align}
where the last equation holds for all $m \leq -1$. Therefore we obtain $b \in S^1$ and hence $B \in \Psi^1$ as well, such that \eqref{factorisation} holds.
\end{proof}

We have established the existence of the factorisation \eqref{factorisation} and now it is time to use it to prove facts about the DN map. The following claim is analogous to Proposition 1.2 in \cite{leeuhlmann} -- the main difference is that now we are using matrix valued pseudodifferential operators, so we need to make sure that appropriate generalisations hold.

\begin{prop}\label{PDO}
The DN map $\Lambda_{g, A, Q}$ is a $\mathbb{C}^{m \times m}$-valued pseudodifferential operator of order one on $\partial M$ and satisfies $\Lambda_{g, A, Q} \equiv -B|_{\partial M}$ modulo smoothing.
\end{prop}
\begin{proof}
Assume without loss of generality that $A$ satisfies condition \eqref{cond2} (see the paragraph after this Proposition). Let us take $f \in H^{\frac{1}{2}}(\partial M; \mathbb{C}^m)$ and $u \in \mathcal{D}'(M; \mathbb{C}^m)$ that solves the Dirichlet problem $\mathcal{L}_{A, Q}u = 0$ with $u|_{\partial M} = f$. Then by Lemma \ref{factorisationlemma} we obtain the following equivalent local system:
\begin{align}
(D_{x^n}\times Id + iB)u &= v \quad \text{with} \quad u|_{x^n = 0} = f\\
(D_{x^n}\times Id +iE \times Id - iB)v &= h \in C^\infty([0, T] \times \mathbb{R}^{n-1}; \mathbb{C}^m) \label{eqn4}
\end{align}
for some $T>0$ and a local coordinate system $x' = (x^1, \dotso, x^{n-1})$ at $\partial M$. By \eqref{eqn4} and Remark 1.2 from Treves \cite{treves}, we may furthermore assume that $u \in C^\infty([0, T]; \mathcal{D}'(\mathbb{R}^{n-1}; \mathbb{C}^m))$.

Then writing $t = T - x^n$, we may view the equation \eqref{eqn4} as backwards generalised heat equation:
\begin{align*}
\partial_tv - (B - E \times Id)v = -ih
\end{align*}
and by standard elliptic interior regularity we obtain that $u$ is smooth and hence, so is $v|_{x^n = T}$. Since the principal symbol of $B$ is negative, by Lemma \ref{WPHEQN} it satisfies Condition \ref{condtreves} (the basic hypothesis of a well-posed heat equation -- see Section \ref{app:A} for more details) and so the solution operator for this equation is smoothing by Theorem 1.2 in Chapter 3 of \cite{treves}. Thus $v \in C^{\infty}([0, T] \times \mathbb{R}^{n-1}; \mathbb{C}^m)$.

Let us set $Rf := v|_{\partial M}$ -- the above argument shows $R$ is a smoothing operator and also $D_{x^n}u|_{\partial M} = -iBu|_{\partial M} + Rf$. Therefore $\partial_{x^n}u|_{\partial M} \equiv Bu|_{\partial M}$ modulo smoothing, which proves the claim.
\end{proof}

The final step in this procedure is to express the Taylor series of $g$, $A$, $q$ in terms of the symbols $\{b_j \mid j \leq 1 \}$ that we obtained in Proposition \ref{PDO}. However, before proving such a result, we need to ``normalise" the metric and the connection -- here we refer to our Lemma \ref{lemma} and to Lemma 2.1 (b) from \cite{conformalcalderon}: there exists an automorphism $F$ of $E$ such that $F|_{\partial M} = Id$ and a positive function $c$ on $M$, with $c|_{\partial M} = 1$ and $\partial_\nu c|_{\partial M} = 0$ ($\nu$ is the outer normal) such that $\tilde{A} = F^*(A)$ and $\tilde{g} = c^{-1}g$ satisfy:
\begin{align}\label{cond1}
\tilde{\partial}_{x^n}^j(\tilde{g}_{\alpha \beta} \tilde{\partial}_{x^n}\tilde{g}^{\alpha \beta})(x', 0)& = 0 \quad \text{for} \quad j \geq 1\\
\tilde{A}_n (x', \tilde{x}_n)&= 0 \label{cond2}
\end{align}
where by $(x', \tilde{x}^n)$ we have denoted the $\tilde{g}$-boundary normal coordinates and $\tilde{\partial}_{x^n}$ denotes $\partial_{\tilde{x}^n}$; \eqref{cond2} holds for all sufficiently small $\tilde{x}^n$, i.e. in a neighbourhood of the boundary. Also notice that the condition \eqref{cond1} is equivalent to $\mathcal{L}^j_{\tilde{N}} \tilde{H}|_{\partial M} = 0$ for $j \geq 1$, as stated in \cite{conformalcalderon}; here $\tilde{N} = \tilde{\partial}_{x^n}$, $\mathcal{L}$ is the Lie derivative and $\tilde{H}$ is the mean curvature of the hypersurfaces given by setting $\tilde{x}^n$ equal to constant. Then by the invariance property of the DN map, we have $\Lambda_{g, A, Q} = \Lambda_{\tilde{g}, \tilde{A}, \tilde{Q}}$ for $Q_c = c^{\frac{n-2}{4}} \Delta_g(c^{-\frac{n-2}{4}}) \times Id$ and $\tilde{Q} = c(F^{-1}QF + Q_c) = c(F^*(Q) + Q_c)$. We will call a triple $\{g, A, Q\}$ that satisfies conditions \eqref{cond1} and \eqref{cond2} \emph{normalised}. Moreover, we will use the notation $f_1 \simeq f_2$ to denote that $f_1$ and $f_2$ have the same Taylor series (as in \cite{LCW}).

\begin{theorem}\label{boundarydet}
Assume $M$ satisfies $\dim{M} = n \geq 3$ and the triple $\{g, A, Q\}$ is normalised. Let $W \subset \partial M$ open, with a local coordinate system $\{x^1, \dotso, x^{n-1}\}$ and let $\{b_j \mid j \leq 1 \}$ denote the full symbol of $B$ (see Lemma \ref{factorisationlemma}) in these coordinates. At any point $p \in W$, the full Taylor series of $g$, $A$ and $Q$ can be determined by the symbols $\{b_j\}$ by an explicit formula.

In particular, if $\Lambda_{g_1, A_1, Q_1} = \Lambda_{g_2, A_2, Q_2}$ and we assume that $\{g_i, A_i, Q_i\}$ are normalised for $i = 1, 2$, then $g_1 \simeq g_2$, $A_1 \simeq A_2$ and $Q_1 \simeq Q_2$. Moreover, if $\Lambda_{g_1, A_1, Q_1} = \Lambda_{g_2, A_2, Q_2}$ and $g_1 \simeq g_2$ on all of $\partial M$, then we also have $\tilde{A}_1 \simeq \tilde{A}_2$ and $\tilde{Q}_1 \simeq \tilde{Q}_2$, for $\tilde{A}_i = F_i^*(A_i)$ and $\tilde{Q}_i = F_i^*(Q_i)$ for $i = 1, 2$; here $F_i$ are automorphisms of $E$ satisfying $F_i|_{\partial M} = Id$ and such that $\tilde{A}_i$ satisfy condition \eqref{cond2} for $i = 1, 2$.
\end{theorem}
\begin{proof}
 Since we have:
\begin{align*}
\partial_{x^n}g_{\alpha \beta} = - (g_{\alpha \rho} \partial_{x^n}g^{\rho \gamma}) g_{\gamma \beta}
\end{align*}
it suffices to determine the inverse matrix $g^{\alpha \beta}$ and its normal derivatives. By the formula \eqref{b1}, we have that $b_1^2 = -g^{\alpha \beta} \xi_\alpha \xi_\beta$ determines $g^{\alpha \beta}|_{\partial M}$.

If we write $\omega = \frac{\xi'}{|\xi'|_g}$ and use the notation:
\begin{align*}
k^{\alpha \beta} = \partial_{x^n}g^{\alpha \beta} - (g_{\gamma \delta} \partial_{x^n}g^{\gamma \delta}) g^{\alpha \beta}
\end{align*}
then we may rewrite \eqref{b0} as follows:
\begin{align*}
b_0 = ig^{\alpha \beta} A_\alpha \omega_\beta - \frac{1}{4}k^{\alpha \beta} \omega_\alpha \omega_\beta \times Id + T_0(g^{\alpha \beta})
\end{align*}
where $T_0$ depends only on $g^{\alpha \beta}|_{\partial M}$, which is already explicitly determined.

Thus, by plugging in $\pm \omega$, we may recover $A_\alpha$ and $k^{\alpha \beta}$; it is not hard to see that:
\begin{align*}
k^{\alpha \beta} g_{\alpha \beta} = (2 - n)\partial_{x^n}g^{\alpha \beta} g_{\alpha \beta}
\end{align*} 
and we may therefore write:
\begin{align}\label{n=2boundarydeteqn}
\partial_{x^n}g^{\alpha \beta} = k^{\alpha \beta} + \frac{1}{2 - n} (k^{\rho \gamma}g_{\rho \gamma})g^{\alpha \beta}
\end{align}
In the next step we will use the notation $l^{\alpha \beta} = \frac{1}{4}\partial_{x^n}k^{\alpha \beta} + Qg^{\alpha \beta}$. Then we may rewrite \eqref{b-1} as:
\begin{align*}
b_{-1} = \frac{1}{2 \sqrt{q_2}}(ig^{\alpha \beta}(\partial_{x^n}A_\alpha)\omega_\beta - l^{\alpha \beta} \omega_\alpha \omega_\beta) + T_1(g^{\alpha \beta}, \partial_{x^n}g^{\alpha \beta}, A_\alpha)
\end{align*}
where $T_1$ is an expression that depends only on $g^{\alpha \beta}$, $\partial_{x^n}g^{\alpha \beta}$ and $A_\alpha$ which have already been explicitly determined. Therefore, we may recover $l^{\alpha \beta}$ and $\partial_{x^n}A_\alpha$. Now, inductively we may prove the formula:
\begin{multline*}
b_{m-1} = \Big(\frac{1}{2\sqrt{q_2}}\Big)^{m-1}(ig^{\alpha \beta} \partial_{x^n}^{|m-1|} A_\alpha \omega_\beta - \partial_{x^n}^{|m|}l^{\alpha \beta} \omega_\alpha \omega_\beta)\\
 + T_{m-1}(g^{\alpha \beta}, \dotso, \partial_{x^{n}}^{|m - 1|} g^{\alpha \beta}, A_\alpha, \dotso, \partial_{x^n}^{|m|} A_\alpha, Q, \dotso, \partial_{x^n}^{|m + 1|}Q)
\end{multline*}
for $m \leq -1$, where $T_{m-1}$ only depends on the quantities in the bracket. Therefore, by induction we may explicitly determine $\partial_{x^n}^j l^{\alpha \beta}$ and $\partial_{x^n}^j A_\alpha$ for all $j \geq 0$.

Finally, we claim that we may inductively recover $\partial_{x^n}^{j+2}g^{\alpha \beta}$ and $\partial_{x^n}^jQ$ for any $j \geq 0$; let us also denote $S_j = g_{\alpha \beta} \partial_{x^n}^jg^{\alpha \beta}$. For the base case $j = 0$, notice that $\partial_{x^n}(g_{\alpha \beta}\partial_{x^n}g^{\alpha \beta}) = 0$, which implies that $S_2 = - \partial_{x^n}g_{\alpha \beta} \partial_{x^n}g^{\alpha \beta}$, i.e. we know $S_2$.

Therefore, since we know $l^{\alpha \beta}$, we may also explicitly determine $\frac{1}{4}\partial^2_{x^n}g^{\alpha \beta} \times Id + Qg^{\alpha \beta} =: P_0^{\alpha \beta}$. This implies:
\begin{align*}
P_0^{\alpha \beta} g_{\alpha \beta} = (n - 1)Q + \frac{1}{4}S_2
\end{align*}
from which we easily infer the knowledge of $Q$ and hence also of $\partial^2_{x^n}g^{\alpha \beta}$.

For the inductive step, we may do something very similar: we have that for $j \geq 1$, the quantity $P_j^{\alpha \beta} = \frac{1}{4}\partial_{x^n}^{j + 2}g^{\alpha \beta} + (\partial_{x^n}^j Q) g^{\alpha \beta}$ is determined, since the condition $\partial^{j +1}_{x^n}(g_{\alpha \beta} \partial_{x^n}g^{\alpha \beta}) = 0$ determines $S_{j + 2}$ by previously reconstructed quantities. Then by the formula:
\begin{align*}
P_j^{\alpha \beta} g_{\alpha \beta} = (n - 1)\partial_{x^n}^j Q + \frac{1}{4}S_{j+2}
\end{align*}
we may determine $\partial_{x^n}^j Q$ and thus, $\partial^{j +2}_{x^n}g^{\alpha \beta}$ as well. This completes the proof of the induction and of the theorem, since two formal expansions of the same operator in terms of classical symbols that agree modulo $S^{-\infty}$, must also be congruent.
\end{proof}

Let us emphasise that a key role in the above generalisations to the vector case is played by the fact that the operator $d_A^*d_A + Q$ has a principal symbol that is a scalar multiple of identity; the necessary algebra then unveils in much the same way as in the scalar case. A couple of remarks are in place.

\begin{rem}[Boundary determination for surfaces]\rm\label{n=2boundarydet}
There are a few reasons to exclude the case $\dim M = 2$ in Theorem \ref{boundarydet}. To start with, after the proof of Proposition 1.3 in \cite{leeuhlmann}, the authors (considering the case $E = M \times \mathbb{C}$, $A = 0$ and $Q = 0$) remark that all the symbols of $B$ satisfy $b_j = 0$ for $j \leq 0$ (easily checked for $b_0$ by direct computation and for the rest by induction); in other words, if we choose $b_1 = -\xi_1 \sqrt{g^{11}}$, the factorisation \eqref{factorisation} becomes a factorisation into honest differential operators where $B = -\sqrt{g^{11}}D_{x^1}$, which is in compliance with the additional conformal symmetry of the Calder\'{o}n problem for surfaces. Secondly, the equation \eqref{n=2boundarydeteqn} clearly fails to hold when $n = 2$ -- in that case $k^{11} = 0$ clearly so there is no extra information from this expression. However, when we introduce a connection and a potential, one can show that (choose $b_1 = -\xi_1 \sqrt{g^{11}}$ again):
\begin{align*}
b_0 &= i\sqrt{g^{11}}A_1\\
2 \xi_1 b_{-1} &=  \partial_{x^2} A_1 - \big(\partial_{x^1}\sqrt{g^{11}}\big)A_1 - \frac{Q}{\sqrt{g^{11}}}
\end{align*}
Thus, the DN map determines the values of $g_{11}$ and $A_1$ at the boundary (recall that $A_2 = 0$ in a neighbourhood of the boundary). Therefore, we may also determine $\partial_{x^2} A_1 - \frac{Q}{\sqrt{g^{11}}}$ from the expression for $b_{-1}$ and so if $Q = 0$, we determine the normal derivative of order one $\partial_{x^2} A_1$ -- we will need this fact for a later application. If we go on to compute $b_{-2}$, we see that it suffices to determine $\partial_{x^2}g_{11}|_{\partial M}$ to compute derivatives $\partial^j_{x^2}A_1|_{\partial M}$ of all orders $j \geq 2$; however, again, we know we cannot possibly determine $\partial_{x^2}g_{11}|_{\partial M}$ due to the additional conformal symmetry of the problem in two dimensions.
\end{rem}

\begin{rem}[Local boundary determination]\label{localboundarydet}\rm
If we assume that $\Gamma \subset \partial M$ is open and $\Lambda_{g_1, A_1, Q_1}(f)|_{\Gamma} = \Lambda_{g_2, A_2, Q_2}(f)|_{\Gamma}$ for all $f \in C_0(\Gamma)$ and that the coresponding quantities are normalised, then by the local nature of the above argument in Theorem \ref{boundarydet}, we have that: $g_1|_{\Gamma} \simeq g_2|_{\Gamma}$, $A_2|_{\Gamma} \simeq A_2|_{\Gamma}$ and $Q_1|_{\Gamma} \simeq Q_2|_{\Gamma}$.
\end{rem}

We end this chapter with an observation that what we proved so far may be translated to the setting of an arbitrary vector bundle $E$ over $M$, rather than just the trivial one.

\begin{rem}[The case of $E$ topologically non-trivial]\rm \label{boundarydettoponontriv}

Firstly, observe that the factorisation \eqref{factorisationlemma} and so Lemma \ref{factorisationlemma} generalises to this case -- the construction that is performed there is independent of the fact that $A_n = 0$, by standard arguments of construction of global PDOs. So we obtain a first order PDO $B$ acting on sections and the local calculations in Lemma \ref{factorisationlemma} (equations \eqref{b0}, \eqref{b-1}, \eqref{bm}) carry over in the trivialisation where $A_n = 0$. Therefore, by the proof of Proposition \ref{PDO}, we have $\Lambda_{g, A, Q} \equiv -B|_{\partial M}$ modulo smoothing.

Our main result of the chapter, Theorem \ref{boundarydet}, remains valid in the following form. By Lemma \ref{lemma} we may assume that $(A - B)(\frac{\partial}{\partial x^n}) = 0$ in a neighbourhood of the boundary. For a coordinate chart $W \subset \partial M$ and some given trivialisation of $E|_{W}$, we may extend this trivialisation to a neighbourhood $W \times [0, \epsilon)$ of $W$ in $M$. Again, by the proof of Lemma \ref{lemma} we may change the trivialisation by a gauge transformation such that $A_n = B_n = 0$ locally. Then the extraction of the Taylor series from the full symbol of $B$ works the same as before and we have the full jet of $(A - B) \in \Omega^1(M; \text{End }E)$ vanishing at the boundary.

Remarks \ref{n=2boundarydet} and \ref{localboundarydet} clearly generalise to this setting.
\end{rem}

\section{Recovering a Yang-Mills connection for $m = 1$}

In this section we consider the main conjecture in the special case of Yang-Mills connections. We prove Theorem \ref{linebundle} for line bundles in the smooth category. In the proofs, we introduce a new technique that we call ``drilling", based on the degenerate unique continuation principles for elliptic systems -- heuristically, the idea is to gauge transform a pair of connections using suitable gauges to a pair of connections that are singular over a countable union of hypersurfaces and apply the degenerate UCP possibly infinitely many times to ``drill through" the hypersurfaces.

We fix a Yang-Mills connection $A$ on the Hermitian vector bundle $E = M \times \mathbb{C}^m$ (with the standard metric) over a compact Riemannian manifold $(M, g)$ with boundary. Let us extend the connection $A$ to a ``new connection" on the endomorphism bundle $\text{End }E = M \times \mathbb{C}^{m \times m}$ by simply asking that $d_{\tilde{A}} F = dF + AF$ globally, where $\tilde{A}$ is the matrix of one forms with values in $\text{End}(\mathbb{C}^{m \times m})$ induced by $A$ by multiplication on the left; it is easy to check this is a unitary connection. Note that $d_{\tilde{A}}$ \emph{does not} satisfy the usual Leibnitz rule such does the usual connection $D_A F = dF + [A, F]$ on the endomorphism bundle. 

Recall that the DN maps associated to the vector bundle $E$ and operators $\Lapl_A$ and $\Lapl_B$ are equal if and only if they agree for the induced operators $\Lapl_{\tilde{A}}$ and $\Lapl_{\tilde{B}}$ on the endomorphism bundle. Here and throughout the chapter, we will use the same notation $d_A = d + A$ for both covariant derivatives $d_A$ and $d_{\tilde{A}}$, which will hopefully be clear from context. The complex bilinear form on matrix valued $1$-forms $(\alpha, \beta) = g^{ij} \alpha_i \beta_j$ is obtained by extending the usual inner product on forms. 

We start by writing down a simple, but key lemma that will yield the right gauge in our situation:

\begin{lemma}\label{gauge_constr} If $U \subset \mathbb{R}^n$ open and $F: U \to \mathbb{C}^{m \times m}$ is an invertible matrix function and we put $A' = F^*(A)$ for $A$ a matrix of one forms on $U$, then $F$ satisfies $d_A^*d_A F = 0$ if and only if $d^*A' = Q_0(x, A')$, where $Q_0$ is smooth of order zero and quadratic in $A'$, given by $Q_0(x, A') = (A', A')$. If in addition $A$ is Yang-Mills and unitary, then $A'$ satisfies an elliptic non-linear equation with diagonal principal part.
\end{lemma}
\begin{proof}
By using that $d_{A'} = F^{-1}d_A F$ and similarly $d_{A'}^* = F^{-1}d_{A}^* F$, note that $d^*_Ad_A F = 0$ is equivalent to the following:
\begin{gather*}
F F^{-1}d^*_AF F^{-1}d_A F = 0 \iff F d^{*'}_{A'}d_{A'} (Id) = 0 \iff d^*A' = (A', A') \iff d^{*'}_{A'} A' = 0
\end{gather*}
by expanding the $d_{A'}^{*'}d_{A'}$ operator by \eqref{expansion}. Here $^{*'}$ denotes the adjoint w.r.t. the pulled back Hermitian structure by $F$. If $A$ is Yang-Mills and unitary, then by adding $(D_{A'})^{*'}F_{A'} = 0$ to $dd^*{A'} = d\big(Q_0(x, A')\big)$ we get an elliptic system with principal part equal to $dd^* + d^*d$.
\end{proof}


By standard elliptic theory and the fact that $\ker (d^*_A d_A) = \{0\}$, we know that we may solve $d_A^*d_AF = 0$ in $H^1(M; \mathbb{C}^{m \times m})$ uniquely for any boundary condition in $H^{\frac{1}{2}}(\partial M; \mathbb{C}^{m \times m})$ (see Appendix A in \cite{cek3}). Therefore, at least near the boundary, we know that $A'$ exists if $F|_{\partial M}$ is smooth non-singular and that it satisfies the equation $d^* A' = Q_0(x, A')$. Thus we may obtain the following result:

\begin{theorem}\label{local}
Consider two Yang-Mills connections $A$ and $B$ on $E = M \times \mathbb{C}^m$ with the same DN map on the whole of $\partial M$. Then there exists a neighbourhood $U$ of the boundary and a bundle isomorphism $H$ for the restricted bundle $E|_{U}$ with $H|_{\partial M} = Id$ such that $H^*{B} = A$ on $U$. Moreover, if $A$ and $B$ are unitary (with respect to the standard structure), then we have $H$ to be a unitary automorphism.
\end{theorem}
\begin{proof}
By the construction above, we obtain smooth gauge equivalences $F$ and $G$, which satisfy $d_A^*d_A F = 0$ and $d_B^*d_B G = 0$ respectively, with boundary conditions $F|_{\partial M} = G|_{\partial M} = Id$. This is non-singular near the boundary and the connections $A' = F^*(A)$ and $B' = G^*(B)$ satisfy the equations
\begin{align}\label{gauge}
d^* A' = Q_0(x, A') \quad \text{and} \quad d^*B' = Q_0(x, B')
\end{align}
Now we can also expand the equations $(D_{A'})^*F_{A'} = 0 = (D_{B'})^*F_{B'}$ (note that $A'$ and $B'$ are now Yang-Mills with respect to the fibrewise inner product pulled back by $F$ and $G$ respectively, rather than the standard inner product):
\begin{align*}
(d^*d + P) A' = 0 \quad \text{and} \quad (d^*d + P) B' = 0
\end{align*}
where $P$ is a first order, non-linear operator arising from the equality 
\begin{align*}
(d^*d + P)A' = (-1)^{n+1}\star D_{A'} \star F_{A'}
\end{align*}
where $\star$ is the Hodge star extended to bundle valued forms. Therefore by simply applying the operator $d$ to \eqref{gauge} and adding to the Yang-Mills equations, we obtain an \textit{elliptic system of equations}, with diagonal principal part
\begin{align}\label{eqnnn}
\Delta A' = (dd^*+d^*d)A' = Q_1(x, A', \nabla A')
\end{align}
where $Q_1$ is a smooth term of first order, polynomial in $A'$ and $\nabla A'$. In order to use uniqueness of solutions to such equations, we need some boundary conditions -- this is where we use the DN map hypothesis. Without loss of generality, assume that the normal components of connections $A$ and $B$ near the boundary vanish (see Lemma \ref{lemma}).

Thus from equality of the DN maps, we have $\frac{\partial (F - G)}{\partial \nu}|_{\partial M} = 0$. By subtracting the initial equations for $F$ and $G$, we get:
\begin{align}\label{eqn3}
\Delta (F - G) - 2(A, dF) + 2(B, dG) + (d^*A) F - (d^*B) G - (A, AF) + (B, BG) = 0
\end{align}
and the point is that we have $\Delta(F - G)$ equal to lower order terms, where we are fixing the semi-geodesic boundary coordinates $(x, t)$ with $t$ denoting the direction of the normal -- this is because we already know that $(A - B) = O(t^{\infty})$, if $n \geq 3$, by the boundary determination result Theorem \ref{boundarydet}, and $(F - G) = O(t)$. Thus when expanding the Laplacian, we are left with only $\frac{\partial^2}{\partial t^2}$ factor, which then allows us to conclude inductively $(F - G) = O(t^{\infty})$ by repeated differentiation. 

If $n = 2$, notice that by Remark \ref{n=2boundarydet} we have $(A - B) = O(t)$; by \eqref{eqn3} we have $(F - G) = O(t^2)$ and thus we have also that $(A' - B') = O(t)$. Therefore by the elliptic equation \eqref{eqnnn}, the analogous counterpart of it for $B'$ and repeated differentiation we obtain $(A' - B') = O(t^\infty)$. 

Therefore, we are left with two connections $A'$ and $B'$ which satisfy an elliptic equation and have the same full Taylor series at the boundary -- by the unique continuation property for elliptic systems with diagonal principal part (see e.g. Theorem 3.5.2. in \cite{isakov}), we conclude $A' \equiv B'$ in $U$ and hence if we put $H = GF^{-1}$ we have $H^*B = A$ on $U$.

Finally, if $A$ and $B$ are unitary, we have that (locally, in a unitary trivialisation) $H^*(A) = B$ implies by definition that $dH = HB - AH$ and $d(H^*) = -BH^* + H^*A$, by the unitary property of connection matrices -- combining the two, we have:
\begin{align*}
d(HH^*) = [HH^*, A]
\end{align*}
where $[\cdot, \cdot]$ is the commutator. This first order system has a unique solution, which is given by $HH^* = Id$, as $H|_{\partial M} = Id$ and thus $H$ is unitary whenever $H^*(A) = B$.
\end{proof}


The next step is to go \textit{inside} the manifold from the boundary. Namely, the main problem lies in the fact that $F$ can be singular on a large set, stopping our argument of unique continuation. However, at least in the scalar case, we may get over this, by essentially knowing facts about zero sets of solutions to elliptic systems of equations. We need to recall the following definition:

\begin{definition}
A subset of a smooth manifold is called \textit{countably $k$-rectifiable} if it is contained in a countable union of smooth $k$-dimensional submanifolds.
\end{definition}

The result we will need is essentially proved in \cite{bar}, Theorem 2, for the scalar case; the vector case we will need follows in a straightforward manner from its method of proof. We outline it here for completeness.

\begin{lemma}\label{zeroset}
Let $(M_0, g_0)$ be a smooth $n$-dimensional, connected Riemannian manifold. Let $L: C^{\infty}(M_0, \mathbb{R}^l) \to C^{\infty}(M_0, \mathbb{R}^l)$ be a differential operator on vector functions for $l$ a positive integer, such that:
\begin{align*}
Lu(x) = \Delta u (x) + R(x, u(x), du(x))
\end{align*}
where $\Delta$ is the metric Laplacian, $R$ is a smooth function with values in $\mathbb{R}^l$. Moreover, we assume that $R$ \textit{respects the zero section}, i.e. $R(x, 0, 0) = 0$.

Now assume $u \not \equiv 0$ is a solution to $Lu = 0$. Let us denote $\mathcal{N}(u) = u^{-1}(0)$ the zero set and by $\mathcal{N}_{crit}(u) = \mathcal{N}(u) \cap \{x \mid du(x) = 0\}$ the critical zero set. Then we claim that $\mathcal{N}(u)$ is countably $(n - 1)$-rectifiable and moreover, $\mathcal{N}_{crit}(u)$ is countably $(n-2)$-rectifiable.
\end{lemma}
\begin{proof}
Consider the vector bundle $E_0 = \bigoplus_j \Big(\Lambda^jT^* M_0 \otimes \mathbb{R}^l\Big)$ of vector valued differential forms. It is a well known fact that the operator $d + \delta$ is a Dirac operator on the bundle of differential forms with respect to the Riemannian inner product (it respects the Clifford relations), where $\delta$ is the codifferential. Moreover, we have that $(d + \delta)^2 = d\delta + \delta d = \Delta$ on differential forms. Let us consider the operator:
\begin{align*}
V\Big(\sum \omega_i\Big) = R(x, \omega_0, \omega_1) - \omega_1
\end{align*}
where $\omega_i$ is the component of $\omega$ in $\Lambda^iT^* M_0 \otimes \mathbb{R}^l$. Clearly $V$ is smooth and respects the zero section.

Thus, if $Lu = 0$, then $\omega = u + du \in C^{\infty}(M; E_0)$ solves $(d + \delta + V)(\omega) = 0$. The first order operator $D = d + \delta + V$ is a Dirac operator acting on sections of $E_0$, so the Corollary 1 of \cite{bar} applies (the strong unique continuation property holds for a Dirac operator, i.e. we cannot have a non-zero solution vanishing to an infinite order at a point). Thus we get the result for the $\mathcal{N}_{crit}(u) = \mathcal{N}(\omega)$.

Finally, since $D$ has the SUCP, we know that $\mathcal{N}(u)$ consists of points where $u$ vanishes to finite order and hence the Lemma 3 from \cite{bar} applies.
\end{proof}


We are now ready to prove the main theorem:

\begin{proof}[Proof of Theorem \ref{linebundle}]
Firstly, gauge transform both $A$ and $B$ such that the normal component of the connection near the boundary is zero (apply Lemma \ref{lemma}). Consider the gauge constructed in Theorem \ref{local}, i.e. $d_A^*d_A f = 0$ and $d_B^*d_B g = 0$ with the following boundary conditions: $f|_{\partial M} = g|_{\partial M}$, $f|_{V} = g|_{V} = 1$ and $f, g$ have compact support at the boundary contained in $\Gamma$. Here $V \subset \overline{V} \subset \Gamma$ is some non-empty, connected, open subset of $\Gamma$.\footnote{We will actually see later that it is enough to have any $f$ and $g$ non-zero and equal at the boundary.} Let us define $h = \frac{f}{g}$ on the complement of the closed set $\mathcal{N}(g) = g^{-1}(0)$. We furthermore split the zero set into the critical set $\mathcal{N}_{crit}(g) = \mathcal{N}(g) \cap \{x \in M \mid dg(x) = 0\}$ and its complement in $\mathcal{N}(g)$, $S = \mathcal{N}(g) \cap \{x \in M \mid dg(x) \neq 0\}$. 

Now we consider the connections $A' = f^*(A)$ and $B' = g^*(B)$ near the set $V$, where we know $f$ and $g$ are non-zero, so these connections are well-defined. Following the recipe from before, by boundary determination and unique continuation we know that in a neighbourhood of $V$ in $M$, we have $A' \equiv B'$ and thus on this set we also have $B = h^*(A)$ or equivalently
\begin{align}\label{eqn1}
dh = (B - A)h
\end{align}
Notice that $B = h^*(A)$ holds in the connected component $R$ of $V$ in the set $M \setminus \mathcal{N}(g) \cap M \setminus \mathcal{N}(f)$. Notice also that $d(|h|^2) = 0$ on this component by using \eqref{eqn1}, since $A$ and $B$ are unitary, so $|h|$ is constant and hence bounded on this set. This implies that the zero sets of $f$ and $g$ agree as we approach the boundary of $R$. The problem now is how to go further inside the interior of the manifold and go past the zero sets of $f$ and $g$. We will do this by a procedure we call ``drilling holes".

Let us describe this procedure. Firstly, we have that the zero set of $g$ lying in the interior of $M$ is contained in a countable union of codimension $1$ submanifolds by Lemma \ref{zeroset}; denote these manifolds by $M_1, M_2, \dotso$. Consider the following situation: we are given a point $p$ such that we have $g(p) = 0$ and $dg(p) \neq 0$ and moreover, we have $g^{-1}(0)$ locally a hypersurface of codimension one (in this case the rank of $dg$ is equal to one). By going to a tubular neighbourhood of $g^{-1}(0)$ near $p$, we may assume we are in the setting where $g = 0$ in a neighbourhood of zero in the hyperplane $\mathbb{R}^{n-1}$ and the metric satisfies $g_{in} = \delta_{in}$ for $i = 1, 2, \dotso, n$ in this coordinate system. Moreover, assume that we know $dh = h(B - A)$ or equivalently, that $f^*(A) = g^*(B)$, in the region where $\{x_n > 0\}$. Our goal is to extend this equality to the lower part of the space. 

Let us just remark that, in general, the zero set of $g$ can be of codimension one or two, depending on the rank of $dg$; however, if $dg \neq 0$ we anyway know that at least one of $d(\im g) \neq 0$ and $d(\re g) \neq 0$ holds, so the zero set is locally contained in $(\im g)^{-1}(0)$ and $(\re g)^{-1}(0)$, at least one of which is a codimension one submanifold. It can of course happen that the zero of $g$ contains an $(n-1)$-dimensional submanifold, see Figure \ref{zero_set} below for such an example (more precisely, $u$ in this example gives the real part of such a solution, with the imaginary part equal to zero).

By Taylor's theorem we have that $f = x_n f_1$ and $g = x_n g_1$ locally near $0$. Furthermore, $g_1 \neq 0$ in a neighbourhood of $0$ by the assumption and hence $f_1 \neq 0$ as $|h|$ is a non-zero constant in the upper space. We want to consider $A' = f^*(A)$ as before, however $f$ can be zero now and thus $A'$ not well-defined (singular), so we will consider something very similar, i.e. $A'' = x_n A'$ and $B'' = x_n B'$
\begin{align}\label{pomoc1}
A'' &= dx_n + x_n \frac{df_1}{f_1} + x_nA\\
B'' &= dx_n + x_n \frac{dg_1}{g_1} + x_nB\label{pomoc2}
\end{align}
Now both of these are well-defined and the degeneracies have cancelled with $x_n$. Let us rewrite the gauge equations for $A''$ (note that $A'$ is Yang-Mills with respect to pullback inner product by $f^*$) as follows:
\begin{align}\label{YM}
x_n^2 d^*d(A'') + x_n\big(\iota_{\nabla x_n} dA'' - d^*(dx_n \wedge A'')\big) + (-2A'' + 2A_n'' dx_n) = 0\\
x_n d^*(A'') + A_n'' - |A''|^2 = 0
\end{align}
where $A''_n$ is the $dx_n$ component of the $1$-form $A''$. After applying $d$ to the second equation and multiplying with $x_n$, we get the form:
\begin{align}\label{gauge'}
x_n^2 dd^*(A'') + x_n \big(d^* A'' \wedge dx_n + d(A_n'') -  d(|A''|^2)\big) = 0
\end{align}
Now after adding the equation \eqref{YM} to the equation \eqref{gauge'} we get a \textit{degenerate} elliptic second order non-linear equation, which has a diagonal principal part $x_n^2 \Delta$ and every first order term multiplied with $x_n$. Notice also $A'' = B''$ for $x_n > 0$, so $A'' - B'' = O(x_n^{\infty})$ on the hyperplane $x_n = 0$.

By Corollary (11) in \cite{mazzeo}
, we deduce that there exists a unique continuation principle for such equations and hence we obtain $A'' \equiv B''$ in the lower space, by continuing from the hyperplane. More precisely, we may rewrite these non-linear equations for $A''$ and $B''$ in the form
\begin{align*}
x_n^2 \Delta A'' = w(x, A'', \nabla A'') \quad \text{and} \quad x_n^2 \Delta B'' = w(x, B'', \nabla B'')
\end{align*}
where $w$ is a smooth function in its entries. Therefore, after subtracting these two and writing $C'' = B'' - A''$, we may obtain
\begin{multline}\label{lineqn}
x_n^2 \Delta C'' = w(x, B'', \nabla B'') - w(x, A'', \nabla A'')\\ = h_1(x, A'', B'', \nabla A'', \nabla B'') C'' + h_2(x, A'', B'', \nabla A'', \nabla B'') \nabla C''
\end{multline}
by Taylor expanding the $w$ with respect to $C''$; here $h_1$ and $h_2$ are smooth in their entries and act linearly on $C''$ and $\nabla C''$, respectively. Thus, after fixing $h_1$ and $h_2$ as known functions, we may think of \eqref{lineqn} as a linear system of equations (of real dimension $2n$) in $C''$ and thus results from \cite{mazzeo} apply.

Moreover, we have that $h = \frac{f_1}{g_1}$ carries smoothly over the hyperplane and therefore we have $dh = (B - A)h$ by subtracting equations \eqref{pomoc1} and \eqref{pomoc2}, on the other side of the hyperplane. Furthermore, using the relation $d(|h|^2) = 0$ obtained from the gauge equation, we immediately get that $|h|$ is constant and thus, non-zero so we may write $B = h^*(A)$.

Finally, by using Lemma \ref{zeroset} we deduce that $B = h^*(A)$ on the whole connected component (call it $R'$) in $M \setminus \mathcal{N}(g)$ of the points in the lower space in the previously considered chart and therefore, that $h$ is non-zero on $R'$ and that the boundary of $R'$ are the points where (could be empty) $g = 0$. This ends the procedure.

Observe that we may perform this procedure at the boundary for a dense set of points $p \in Q = \Gamma \cap \mathcal{N}(g)$ to extend $h$ such that $h^*(A) = B$ near these points with $h = 1$ on the boundary. In more detail, the set $\{p \in Q \mid dg(p) = 0 \text{ or } df(p) = 0\}$ is small, in the sense that its complement is dense, by Lemma \ref{zeroset}. On this set, near a point $p$, we may use semi-geodesic coordinates and write $f = x_nf_1$ and $g = x_n g_1$ as before; then $h = \frac{f_1}{g_1}$ extends $h$ smoothly and $h = 1$ on boundary, since the DN maps agree. The boundary determination result applied to quantities $A''$ and $B''$ defined in \eqref{pomoc1} and \eqref{pomoc2} and the degenerate unique continuation result of Mazzeo now applies to equations \eqref{YM} and \eqref{gauge'}, to uniquely extend from $\partial M$, as before.

We may now drill the holes and extend $h$ together with the relation $h^*(A) = B$, starting from the component of $V$, where we may use boundary determination. The idea is that drilling the holes connects path components over the possibly disconnecting set $\mathcal{N}(g)$. Let us now give an argument that what we are left with (after drilling the holes) is path connected. 

Let us denote the complement of the zero set $T = M \setminus \mathcal{N}(g)$; obviously $M \setminus (\cup M_i) \subset T$ and $T$ open. Let $x_0 \in M^\circ$ be a point in the open neighbourhood of $V$ where $B = h^*(A)$ and $y$ be any point in $T$. Consider any path $\gamma: [0, 1] \to M$ with $\gamma(0) = x_0$ and $\gamma(1) = y$. We will construct a path $\gamma'$ from $x_0$ to $y$, lying in $T$, by slightly perturbing the path $\gamma$, such that $\gamma$ and $\gamma'$ are arbitrarily close. Let $d$ be the usual complete metric in the space $C^{\infty}([0, 1], M)$ of smooth paths with fixed endpoints $x_0$ and $y$ (see Remark \ref{pathmetric} in the appendix).

By standard differential topology (see \cite{hirsch}), there exists an arbitrarily close path $\gamma_1$ to $\gamma$ (with the same endpoints), such that $\gamma_1$ intersects $M_1$ transversally in a finite number of points $P_1, \dotso, P_k$. There are two possibilities for these points, starting e.g. with $P = P_1$:
\begin{enumerate}
\item There exists a sequence of points $p_i \in M_1$, for $i = 1, 2, \dotso$, converging to $P$, such that $g(p_i) \neq 0$ for all $i$.
\item We have $g = 0$ in a neighbourhood of $P$ in $M_1$ and a sequence of points $q_i \in M_1$ converging to $P$, such that $dg(q_i) \neq 0$.
\end{enumerate}
In the first case, we may slightly perturb $\gamma_1$, such that it goes through one of the points $p_i$ and is sufficiently close in the metric $d$. These are complementary conditions, so if the first item does not hold, then the second one does: in that case, we may still perturb $\gamma_1$ to go through one of the points $q_i$, by the above argument of drilling holes. Notice that each of the points $p_i$ or $q_i$ has a neighbourhood in $M_1$ through which we can perturb the curve and therefore, there exists an $\epsilon > 0$, such that if we move our curve by a distance less than $\epsilon$ in the $d$-metric, the resulting curve will go through this neighbourhood.

Now inductively, we may perform the same procedure for all $j = 1, 2, \dotso, k$ and, each time, taking the perturbations small enough such that it does not interfere with the previously done work -- what we obtain is $\gamma_1'$, which is sufficiently close to $\gamma_1$ and which does not hit $M_1$, minus the deleted holes. Thus we obtain a Cauchy sequence of curves $\gamma_1', \gamma_2', \dotso$ such that $\gamma_i'$ does not hit $M_1, M_2, \dotso, M_i$, minus the deleted holes. Since the space of curves is complete, we obtain a limiting curve $\gamma_i' \to \gamma'$, which lies completely in $T$ together with the drilled holes and furthermore satisfies $d(\gamma, \gamma') < \delta$ for some pre-fixed $\delta > 0$. In particular, this implies that the lengths of the curves are close, i.e. $|l(\gamma) - l(\gamma')| < \delta'$ for some $\delta' > 0$ (here $l$ denotes the length of the curve in the underlying Riemannian manifold). Let us denote the union of all of the drilled holes, i.e. neighbourhoods of some of the points $q_i$ in the item $(2)$ above, by $T_\gamma$.

Moreover, we may repeat the above argument for all paths $\gamma$, now between \emph{any} two points in $T$ -- denote the set of new drilled holes by $S_\gamma$. Then we redefine $T$ as:
\begin{align*}
T:= T  \bigcup_{\gamma\,\mathrm{from \,x_0\, to\, y}} T_{\gamma} \bigcup_{\gamma' \,\mathrm{from\, any\, x\, to\, any\, y}} S_{\gamma'}
\end{align*}
where the first union runs over all of the curves $\gamma$ starting at $x_0$ and ending at $y \in M^{\circ} \setminus \mathcal{N}(g)$; the second one is over all paths $\gamma'$ between points in $M^{\circ} \setminus \mathcal{N}(g)$. It is easy to see that $T \subset M^{\circ}$ is open and connected and furthermore, it satisfies the property that for any curve $\gamma$ between any two points $x, y \in T$, there exists a sequence of curves $\gamma_n$ between $x$ and $y$, lying wholly in $T$, such that $d(\gamma_n, \gamma) \to 0$ as $n \to \infty$. Also, we have $B = h^*(A)$ on $T$ by the argument of drilling holes.

Let us denote by $d_1$ the inherited metric of $T$ as a subspace of $M$ and by $d_2$ the metric in the Riemannian manifold $(T, g|_{T})$. Therefore, as a result of the above construction, we may claim the following about these metrics:\footnote{We just proved that the inherited subspace metric on $T \subset M$ and the path metric as a submanifold of a Riemannian manifold are Lipschitz equivalent with Lipschitz constant equal to $1$.}
\begin{align*}
d_2(x, y) = \inf\{l(\gamma) \mid \gamma \text{ a piecewise smooth path from $x$ to $y$ lying in $T$}\} = d_1(x, y)
\end{align*}
Notice also that we have, by the Fundamental Theorem of Calculus, if $\gamma$ is a path from $x$ to $y$ lying in $T$:
\begin{multline*}
|h(x) - h(y)| = \Big|\int_0^1 dh_{\gamma(t)}(\dot{\gamma}(t)) dt\Big| \leq \int_0^1 \big|\langle{\nabla h_{\gamma(t)},\dot{\gamma}(t)}\rangle \big| dt\\
\leq \int_0^1 |\nabla h_{\gamma(t)}|_g\cdot |\dot{\gamma}(t)|_g dt \leq C \int_0^1 |\dot{\gamma}(t)|_gdt = C \cdot l(\gamma)
\end{multline*}
by Cauchy-Schwarz, 
where $\nabla h$ is the gradient of $h$ and $C$ is a uniform bound on $dh$ (which follows from the global relation $dh = (B - A)h$ in $T$ and uniform bounds on $h$, $A$ and $B$). If we take the infimum over all such curves $\gamma$, we obtain the inequality $|h(x) - h(y)| \leq C d_2(x, y) = C d_1(x, y)$ and therefore obtain that $h$ is Lipschitz and so uniformly continuous over $T$.

Therefore, $h$ can be extended continuously\footnote{Here we are using the elementary fact that a uniformly continuous function can be uniquely continuously extended to the closure of its domain.} to the whole of $M$ and by inductively differentiating the relation $dh = (B - A)h$, we moreover have that all partial derivatives of $h$ can be continuously extended. That these continuous extensions of derivatives are actual derivatives of the extension of $h$ is proved in Lemma \ref{extension_lemma} in the Appendix; see also Remark \ref{rmk_boundary_extension} in the Appendix for the extension to the boundary. This proves $h^*(A) = B$ on the whole of $M$ with $h$ smooth and that $h|_{\Gamma} = 1$; $h$ also unitary. This finishes the proof.
\end{proof}


\begin{rem}[Topological remarks]\rm
One can see that the complement of the disconnecting set $\mathcal{N}(g)$ can indeed have non-trivial topology; this justifies the use of our argument of drilling holes. For simplicity, we will consider real harmonic functions $g$ with $\Delta g = 0$ in the open unit disk. Firstly, one may observe that there are two types of points in $\mathcal{N}(g)$: the critical and the non-critical ones. The non-critical ones are simple: they are locally contained in an analytic curve, whereas the critical ones are isolated (since they are exactly the set of points where $f' = 0$, where $f$ holomorphic and $u = \re{(f)}$) and are locally zero sets of harmonic polynomials, i.e. zero sets of $\re{((z - P)^m)}$, where $m \geq 2$ an integer. Thus at these critical points, $\mathcal{N}(g)$ is a union of $m$ analytic curves meeting at $P$ at equal angles. Also, there are no loops in $\mathcal{N}(g)$, due to the uniqueness of the Dirichlet problem and analytic continuation. Therefore, if $g$ has an analytic extension to the closed disk, there are finitely many components in the complement of $\mathcal{N}(g)$, but if no such extension exists and $g$ is zero at infinitely many points at the boundary, then we may expect infinitely many components. This is because for each such vanishing, non-critical point of $g$ at the boundary we have an ``end" coming inside the disk, which returns to the boundary at some other point, by the analysis above. See Figure \ref{zero_set} for a concrete example and \cite{walker, DCH} for further analysis.
\end{rem}

\begin{figure}[h]
   \centering
    \includegraphics[width=0.6\textwidth]{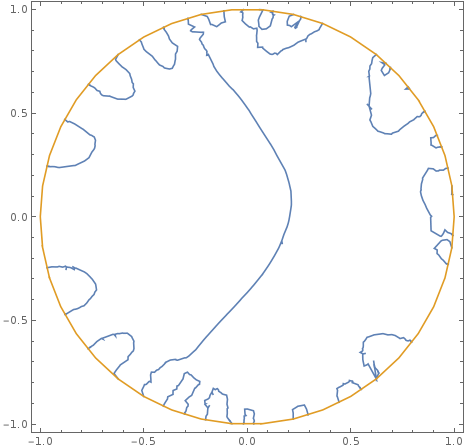}
            \captionsetup{format=hang, labelfont=sc}
    \caption{In blue -- the zero set of the harmonic function with boundary value equal to $f(\theta) = \theta \cdot \sin{\frac{100}{\theta}}$ on the unit disc, where $\theta \in (-\pi, \pi)$ is the angular coordinate. In orange -- the unit circle. The accumulation point is $(1, 0)$.}
\label{zero_set}
\end{figure}

\begin{rem}\rm\label{rem:minimalcondition}
Notice that the only two implications we were using in Theorem \ref{linebundle} from the equality of the DN maps for $A$ and $B$, were:
\begin{enumerate}
\item By boundary determination, the connections $A$ and $B$ have the same full jets at the boundary in suitable gauges.
\item There exist two non-zero solutions $f$ and $g$ to $d_A^*d_A f = d_B^*d_B g = 0$, such that $f|_{\partial M} = g|_{\partial M}$ and $\partial_\nu f|_\Gamma= \partial_\nu g|_\Gamma$ for a non-empty open $\Gamma \subset \partial M$.
\end{enumerate}
We then showed that the quotient $\frac{f}{g}$ is the gauge between $A$ and $B$.
\end{rem}

\begin{rem}[Alternative boundary extension]\rm\label{newrmk_boundary_extension}
A different approach to extension of the gauge to the boundary, by using the partial differential equations that it satisfies (that is $H^{-1}dH + H^{-1}AH = B$), can be found in Proposition 4.7 from \cite{KOP}. There, the authors take $A$ and $B$ to a gauge with no normal component (as in Lemma \ref{lemma}), so that the new gauge $H'$ is independent of the normal variable from the equation it satisfies and can clearly be extended smoothly. Note that the same proof works in the non-unitary case.
\end{rem}

By a careful analysis of the proof of Theorem \ref{linebundle}, we have the following result:

\begin{prop}\label{prop:linebundlenonunitary}
Conclusions of Theorem \ref{linebundle} hold in the case of general $GL(m, \mathbb{C})$ Yang-Mills connections.
\end{prop}
\begin{proof}
We use the same notation as in the original proof. The first issue is that we do not know that $d(|h|^2) = 0$ any more, so a priori $f$ and $g$ need not have the same zero set. We address this as follows.

By gauge transforming $A$ and $B$ locally near a zero set hypersurface $g^{-1}(0)$ of $g$ (or $f^{-1}(0)$ for $f$), we may assume that the $dx_n$ components of $A$ and $B$ vanish and $B = h^*(A)$ in $\{x_n > 0\}$. Then by Remark \ref{newrmk_boundary_extension}, we see that $h$ is independent of $x_n$ and so extends in a non-singular way beyond $\{x_n = 0\}$ -- thus $\frac{f}{g}$ also extends in a smooth and non-zero way. From that point, we may apply the earlier argument in the same way. 

By drilling along paths as before, we are left with $h: T \to \mathbb{C}$ such that $h^*(A) = B$ ($h$ is nowhere zero) and $h|_{\Gamma} = 1$, where $T$ is dense, connected and open and moreover, $T$ satisfies the property that given a curve $\gamma$ in $M$ with endpoints in $T$, there exist arbitrarily close curves to $\gamma$ with the same endpoints, lying wholly in $T$. Notice that $dh = (B - A)h$ on $T$ implies $dA = dB$ by density, which by the proof of Theorem 7.3. from \cite{cek1} immediately proves the claim. Alternatively, we will extend $h$ to a gauge on $M^\circ$ by proving uniform bounds on $h$ on compact subsets of $M^\circ$.

Take a point $p \in M^\circ \setminus T$. Note that we have in $T$
\[dh = (B - A)h\]
So if we take a small ball $U$ around $p$, we have a logarithm $f$ in $U$, by solving $df = B - A$ (such an $f$ exists as $dA = dB$). Then by uniqueness we have $h = c \cdot e^f$ for a constant $c$, as $U \cap T$ is connected. So $h$ extends smoothly on $U$ and by density, we have $h^*(A) = B$ on $U$. So $h$ extends to $M^\circ$, such that $h^*(A) = B$. We are left to observe that Remark \ref{newrmk_boundary_extension} extends $h$ smoothly to $\partial M$.
\end{proof}

\section{Recovering a Yang-Mills connection for $m > 1$ via geometric analysis}

In this section, we prove Theorem \ref{analytic_m>1} for rank $m > 1$ bundles in the analytic category; the analytic assumption is technical -- the main obstacle to solving for rank $m > 1$ in the smooth case is the possibility that the zero set of $\det{F}$ for $F$ satisfying $d^*_Ad_A F = 0$ could potentially be large. For this, it suffices to prove that the determinant does not vanish to an infinite order (if non-trivial) at any point, since by some general theory the zero set is then contained in an $(n-1)$-rectifiable set \cite{bar}. However, due to the recent work of the author \cite{cek2} we have strong evidence and some counterexamples to even the weak unique continuation principle. These counterexample seem not to be generic, so we hope that this method can still be pursued.

In addition to this, we would like to point out that it is no longer true that the critical zero set of $\det{F}$ is $(n-2)$-rectifiable, as in the case $m = 1$; a class of counterexamples is given by $F = \begin{pmatrix} f & 0\\ 0 & f \end{pmatrix}$, where we have that $\mathcal{N}_{crit}(\det{F}) = \mathcal{N}_{crit}(f^2)$ contains the set where $f$ vanishes (since $d(f^2) = 2f df$). Therefore if $f$ vanishes on an $(n-1)$-dimensional set, then the critical set is also $(n-1)$-dimensional. One such example is given by $M = \mathbb{R}^2$ and $f(x, y) = x$ which vanishes along the $y$-axis and solves $\Delta_{\mathbb{R}^2}(x) = 0$.



Therefore, here we consider the case of analytic functions and generalise the proof of Theorem \ref{linebundle}. Analytic functions satisfy the SUCP by definition and in addition, the zero set is given by a countable union of analytic submanifolds of codimension one. This can easily be seen by considering the order of vanishing at a point and by observing that locally, every point in the zero set is contained in $(\partial^{\alpha} h)^{-1}(0)$, where $h$ is the analytic function and $\alpha$ is a multi-index such that $\nabla(\partial^{\alpha} h) \neq 0$.

Note that if $A$ and $g$ are analytic, one has $F$ satisfying $d_A^*d_A F = 0$, which is an elliptic system with analytic coefficients and thus by a classical result of Morrey \cite{morrey} its entries are analytic. Therefore, the determinant is analytic also and thus cannot vanish to the infinite order at a single point, if it is non-trivial. Unless otherwise stated, for the rest of the section $(M, g)$ is a compact analytic (in the interior) Riemannian $n$-manifold with boundary. We first prove a result about the zero set of the determinant of a matrix solution where $A$ is Yang-Mills and not necessarily analytic:

\begin{lemma}\label{recYM-analytic}
Let $E = M \times \mathbb{C}^{m}$ a Hermitian vector bundle and $A$ a unitary Yang-Mills connection on $E$. Then any solution $F: M \to \mathbb{C}^{m \times m}$ to $d_A^*d_A F = 0$ with $\det{F}$ non-zero has $\mathcal{N}(\det{F})$ to be $(n-1)$-rectifiable. Moreover, $\det{F}$ satisfies the strong unique continuation property.
\end{lemma}
\begin{proof}
This is a local result, so assume we have a point $p \in M^\circ$ in the interior and take a small coordinate ball $B^n(\epsilon)$ around $p$, such that $\lVert A \rVert_{L^{n/2}(B^n(\epsilon))}$ is small enough; by a dilation we may also assume $B^n$ is the unit ball and we also have $\lVert A \rVert_{L^{n/2}(B^n)}$ stays the same as for the smaller ball, by a computation. By a result of Uhlenbeck \cite{uhlenbeck}, we have a gauge automorphism $X: B^n \to U(m)$ that takes $A$ to $A' = X^*(A)$ with $d^*(A') = 0$. In this particular gauge, the Yang-Mills equations become elliptic and therefore, $A'$ is analytic.

Similarly, since $d_A^*d_A F = 0$, we have $d_{A'}^*d_{A'} F' = 0$, where $F' = X^{-1}F$ -- thus $F'$ is also analytic. Moreover, $\det{F'} \det{X} = \det{F}$ and so $\mathcal{N}(\det{F}) = \mathcal{N}(\det{F'})$ on $B^n$, as $X$ is always invertible; since $\det{F'}$ is analytic, we obtain the first part of the result.

Finally, from the relation $\det{F'} \det{X} = \det{F}$ and the fact that $\det{X}$ is non-zero on $B^n$, we immediately get that $\det{F}$ vanishes up to order $k$ if and only if $\det{F'}$ vanishes up to order $k$ -- thus $\det{F}$ satisfies the SUCP, as $\det{F'}$ does.
\end{proof}

This means that on $M^{\circ}$ we have $\mathcal{N}(\det{G}) \subset \cup_i M_i$ for $M_i$ a countable family of analytic submanifolds of codimension one, where $G$ solves $d_B^*d_B G = 0$ and represents the gauge we used in the previous section. Notice that $G^*(B)$ then satisfies an elliptic system (as before), but with analytic coefficients -- therefore $G^*(B)$ is also analytic, but only on the set where $G$ is invertible.

To overcome this, we use the method of proof of the $m = 1$ case, Theorem \ref{linebundle}, and the main difference is that now we will be able to use analyticity to uniquely continue the solution when drilling hyperplanes, whereas before we relied on the unique continuation property of elliptic equations.

\begin{proof}[Proof of Theorem \ref{analytic_m>1}]
Assume we have the gauges $F$ and $G$ that solve $d_A^*d_A F = 0$ and $d_B^*d_B G = 0$ with $F|_{\partial M} = G|_{\partial M}$, $\text{supp}(F|_{\partial M}) = \text{supp}(G|_{\partial M}) \subset \Gamma$ and equal to identity on an open, non-empty subset $V$ of $\Gamma$. Then $F^*(A) = G^*(B)$ in the neighbourhood $U$ of $V$ in the manifold, as in Theorem \ref{local}, by unique continuation; equivalently, we have $H^*(A) = B$ where $H = FG^{-1}$. We also have that $H$ is unitary.



We may suppose that $\mathcal{N}(\det{G}) \subset \cup_i M_i$ for $M_i$ analytic submanifolds of codimension one, by Lemma \ref{recYM-analytic}. Let us now prepare the terrain for drilling the holes -- consider a point $p$ in $M_i$ for some $i$ and assume $\det{G} = 0$ near $p$ in $M_i$, such that the following property holds:
\begin{equation}\label{eqn2}
\frac{\partial^j (\det{G})}{\partial x_n^j} = 0 \quad \text{ for }\quad j=0, 1, \dotso, k-1
\end{equation}
in a neighbourhood of $p$ in $M_i$, for some $k$; we also ask that $\frac{\partial^{k} (\det{G})}{\partial x_n^{k}}(p) \neq 0$. Here we are using the analytic chart given by coordinates on $M_i$ near $p$ and the $x_n$ coordinate given by following the normal geodesics (which is also analytic). We make the standing assumption that $F$ and $G$ are invertible for $x_n > 0$ in this coordinate system and that $F^*(A) = G^*(B)$ in the same set.

This means that near $p$, by Taylor's theorem we have $\det{G} = x_n^k g_1$ for some $g_1$ that satisfies $g_1(p) \neq 0$ -- therefore locally at $p$, $\mathcal{N}(\det{G})$ is contained in $M_i$. Since $H$ is unitary for $x_n > 0$, we have $H = FG^{-1} = \frac{F\adj{G}}{x_n^k g_1}$ is bounded on this set and therefore $F\adj{G} = x_n^k H_1$ for some smooth $H_1$ near $p$ -- we get $H = \frac{H_1}{g_1}$ locally, which means that $H$ extends smoothly to the other side of $M_i$ in the proximity of $p$. Moreover, as $H$ unitary we have $|\det{H}| = 1$ at $p$ and so $H$ is invertible near $p$.

To use the real-analyticity, we must transform $A$ and $B$ such that they are locally analytic -- we do this by constructing the Coulomb gauge automorphisms (unitary) $X$ and $Y$ locally near $p$ such that $A' = X^*(A)$ and $B' = Y^*(B)$ and moreover, that $d^*(A') = d^*(B') = 0$ (by the proof of Lemma \ref{recYM-analytic}). Then $A'$ and $B'$ are analytic as in the previous lemma and moreover, we have $F' = X^{-1}F$ and $G' = Y^{-1}G$ satisfying $d_{A'}^* d_{A'} F' = 0$ and $d_{B'}^*d_{B'} G' = 0$ -- therefore $F'$ and $G'$ are analytic, as well.

Thus we may write $H' = X^{-1}F G^{-1}Y$ and by rewriting $H^*(A) = B$ (by assumption) we get $H'^*(A') = B'$ for $x_n > 0$ in a neighbourhood of $p$. Let us now observe that $H'$ also smoothly (analytically) extends over the hyperplane $M_i$ -- this is because, by Taylor expanding $\det{(Y^{-1}G)} = \frac{\det{G}}{\det{Y}}$, we get
\begin{align*}
H' = X^{-1}F \cdot \frac{\adj\big({Y^{-1}G\big)}}{g' x_n^k}
\end{align*}
where $g' = \frac{g_1}{\det{Y}}$ is analytic, so $g' \neq 0$ near $p$. However, we know $H'$ is bounded near $p$, since $H$, $X$ and $Y$ are. Thus $X^{-1}F \cdot \adj{(Y^{-1}G)} = F' \cdot \adj{G'} = x_n^k H''$ for some analytic $H''$, by looking at the expansion of $F' \adj{(G')}$ -- in conclusion, $H' = \frac{H''}{g'}$ analytically extends near $p$ and $H'$ is also invertible at $p$ as $H$, $X$ and $Y$ are.

Finally, it is easy now to see that $(H')^*(A') \equiv B'$, since both sides are analytic near $p$ and $(H')^*(A') = B'$ for $x_n > 0$; equivalently $H^*(A) \equiv B$ near $p$. This ends the drilling argument and we may repeat the part of the argument of Theorem \ref{linebundle} which perturbs the curve by an arbitrarily small amount so that it goes through the holes.

Let us briefly describe the analogous procedure from Theorem \ref{linebundle}. Take a base point $x_0 \in U \cap M^{\circ}$ and consider a path $\gamma$ lying in the interior, from $x_0$ to some point $y \in M^{\circ}$. We perturb $\gamma$ such that it intersects $M_1$ transversally at $P_1, \cdots P_k$ ($k$ can be zero). At $P_1$, consider the tubular neighbourhood (analytic) given by following geodesics perpendicular to $M_1$. If there exists a sequence of points $p_j \in M_1$ that converges to $P_1$ and $\det{G} \neq 0$ at every $p_j$, we may perturb $\gamma$ slightly and get it to pass through one of the points $p_j$. Otherwise, inductively, since $\det{G}$ satisfies the SUCP by Lemma \ref{recYM-analytic}, there exists a positive integer $k$ such that $\frac{\partial^i (\det{G})}{\partial x_n^i} = 0$ for $0 \leq i \leq k-1$ in a neighbourhood of $P_1$ and there exists a sequence of points $p_j \in M_1$ that converge to $P_1$ and $\frac{\partial^{k} (\det{G})}{\partial x_n^{k}} \neq 0$ at each $p_j$. This leaves us in the setting \eqref{eqn2} from the previous paragraph, suitable for drilling the holes -- inductively, we perturb $\gamma$ such that it intersects the $M_i$ in the drilled holes.

Thus we obtain a smooth (analytic in the interior) extension of $H = FG^{-1}$ to the whole of $M$, such that $H^*(A) = B$ and $H|_{V} = Id$.



To get the wanted gauge with $H|_{\Gamma} = Id$, we will need a slightly different argument, because we do not know if $\det{F}$ and $\det{G}$ vanish to infinite order at the boundary, as we did not assume analyticity up to the boundary. We will construct a sequence of matrix functions $H_i$ such that $H_i^*(A) = B$ and use a compactness argument to take the limit. Consider nested open sets $V_i$, with $\emptyset \neq V_1 \subsetneqq \overline{V}_1 \subsetneqq V_2 \subsetneqq \overline{V}_2 \subsetneqq \dotso \subsetneqq  \Gamma$ with the property $\cup_i V_i = \Gamma$. Construct appropriate $F_i$ and $G_i$ supported in $\Gamma$, such that $F_i|_{V_i} = G_i|_{V_i} = Id$, solving $d_A^*d_A F_i = 0$ and $d_B^*d_B G_i = 0$ and setting $H_i = F_iG_i^{-1}$ -- by the argument above $H_i^*(A) = B$ and $H_i|_{V_i} = Id$. Now the important property that the gauges satisfy is that they are unitary, hence bounded and they satisfy $dH_i = H_i B - A H_i$ so that $dH_i$ are uniformly bounded. By inductively differentiating this relation, we get that all derivatives of $H_i$ are uniformly bounded on $M$. By the Arzel\`{a}-Ascoli theorem (or the Heine-Borel property of $C^{\infty}(M)$) we get a convergent subsequence with a limit $H \in C^{\infty}(M; U(m))$, $H|_{\Gamma} = Id$ and $H^*(A) = B$. This finishes the proof.
\end{proof}

\begin{rem}\rm\label{rem:smoothm>1}
\emph{If} we found a smooth solution $F$ to $d_A^*d_A F = 0$ with $\det F$ not vanishing to infinite order at any point \emph{and} proved that the unique continuation property from a hyperplane holds for degenerate elliptic systems, with degeneracies of the form $x_n^{2k}\Delta_g \times Id + x_n^k F_1 + F_0$, where $F_0$ and $F_1$ are zero and first order matrix operators, respectively and for \emph{all} $k$ positive integers -- then we would be able to prove uniqueness in the $m > 1$ case in the smooth category, by following the proofs of Theorems \ref{linebundle} and \ref{analytic_m>1}. Note that the SUCP property of $\det F$ is analysed in more detail in \cite{cek2}.
\end{rem}

\section{The case of arbitrary bundles via Runge approximation}

Here we cover the case of topologically non-trivial vector bundles, by using a suitable version of Runge approximation to construct the gauges which are non-singular along a curve. The idea is to show that the connections have the same holonomy, hence they are equivalent. This proof uses more information from the DN map than do the proofs in previous two sections, since it relies on Runge approximation (see Appendix \ref{app:B}).

We assume $E$ is a Hermitian vector bundle of rank $m$ over a compact $n$-dimensional manifold $(M, g)$ with boundary, equipped with a unitary connection $A$. Furthermore, let $\Gamma \subset \partial M$ be a non-empty open set. We denote $\Lapl = d_A^*d_A$.

\begin{lemma}\label{lemma:runge}
    Let $K \subset M^\circ$ be an embedding of a closed interval $I$. Then there exist smooth harmonic sections (w.r.t. $\Lapl$) $s_1, \dotso, s_m$ with $\supp(s_i|_{\partial M}) \subset \Gamma$ for $i = 1, 2, \dotso, m$ and such that
    \begin{equation}
        \spann{\{s_1(x), \dotso, s_m(x)\}} = E(x)
    \end{equation}
    for $x \in K$.
    \begin{proof}
        Let $T$ be a tubular neighbourhood of $K$ embedded in $M^\circ$. Then $T \cong I \times B$ ($B$ is the $(n-1)$-dimensional unit ball). We consider a smaller tube $T'$ contained in $T$ such that $K \subset \partial T'$; we smooth out $T'$ a little bit, keeping $K$ on its boundary.
        
        Since $E|_{T'} = T' \times \mathbb{C}^m$ is trivial, we may solve the Dirichlet problem
        \begin{align}
            \Lapl r_i &= 0 \, \text{ in } \, T'\\
            r_i|_{\partial T'} &= e_i
        \end{align}
        for $i = 1, 2, \dotso, m$ and $e_i \in \mathbb{C}^m$ the $i$-th coordinate vector. Fix $\varepsilon > 0$. Then by Corollary \ref{cor:rungecor}, we get a family of smooth solutions $s_i \in C^\infty(\overline{M}; E)$ solving
        \begin{align}
            \Lapl s_i &= 0 \, \text{ in } \, M\\
            \supp(s_i|_{\partial M}) &\subset \Gamma\\
            \lVert{s_i - r_i}\rVert_{C^0(T')} &< \varepsilon
        \end{align}
        If we take $\varepsilon$ small enough, then $\spann\{s_1(x), \dotso, s_m(x)\} = E(x)$ for $x \in K$, by the construction. This finishes the proof.
    \end{proof}
\end{lemma}

In order to apply the previous lemma to a curve at the boundary, we need to slightly extend our domain, with keeping the assumptions about the DN maps.

\begin{lemma}\label{lemma:extension}
    Assume $E$ is also equipped with a unitary connection $B$ such that $\Lambda_A(f)|_{\Gamma} = \Lambda_B(f)|_{\Gamma}$ for all $f \in C_0^\infty(\Gamma; E)$. Then in a suitable gauge and for a point $p \in \Gamma$, there is a (half-ball) neighbourhood $p \in U$ with an extension $\widetilde{U}$ (with $p \in \widetilde{U}^\circ$), which give further extensions: $M \subset \widetilde{M} := M \cup \widetilde{U}$, $\widetilde{E} := E \cup \widetilde{U}\times \mathbb{C}^m$, $\widetilde{\Gamma} :=\partial \widetilde{U} \cup \Gamma \setminus (\Gamma \cap \partial U) \subset \partial \widetilde{M}$, $\widetilde{A}|_{M} = A$, $\widetilde{B}|_M = B$, $\widetilde{g}|_M = g$. Moreover, we ask that $A = B$ in $\widetilde{M} \setminus M$. 
    
    Then we also have that $\Lambda_{\widetilde{A}}(f)|_{\widetilde{\Gamma}} = \Lambda_{\widetilde{B}}(f)|_{\widetilde{\Gamma}}$ for all $f \in C_0^\infty(\widetilde{\Gamma}; \widetilde{E})$.
\end{lemma}
\begin{proof}
    Take $p \in \Gamma$ with a small half-ball trivialising chart $U$ and extend it slightly to $\widetilde{U}$ to form $\widetilde{M}:=M\cup\widetilde{U}$ with some arbitrary smooth extension of the metric. Extend the bundle trivially, i.e. by gluing in $\widetilde{U} \times \mathbb{C}^m$. By Theorem \ref{boundarydet} and Remark \ref{boundarydettoponontriv}, we know in some gauge we have $A$ and $B$ having the same jets on $\Gamma$. Thus there exist extensions $\widetilde{A}$ and $\widetilde{B}$ that agree in $\widetilde{M}\setminus M$.
    
    For the second claim, note that $\Lapl_{\widetilde{A}} = d_{\widetilde{A}}^* d_{\widetilde{A}}$ does not have zero as a Dirichlet eigenvalue and so the DN maps are always well-defined. Let $\widetilde{u} \in H^1(\widetilde{M}; \widetilde{E})$ be a unique solution to $\Lapl_{\widetilde{A}} \widetilde{u} = 0$ with $\widetilde{u}|_{\partial \widetilde{M}} \in \widetilde{H}^{\frac{1}{2}}(\widetilde{\Gamma}; \widetilde{E})$.
    
    Then $u := \widetilde{u}|_{M}$ satisfies $\Lapl_A u = 0$ and by the assumption on the DN maps, there is $v_0 \in H^1(M; E)$ such that $\Lapl_B v_0 = 0$, $u|_{\partial M} = v_0|_{\partial M}$ and $\mathcal{N}_A u|_{\Gamma} = \mathcal{N}_B v_0|_{\Gamma}$.\footnote{Here we denote the covariant normal derivative by $\mathcal{N}_A u := d_A(u)(\nu)|_{\partial M}$.} Therefore, there exists $\varphi \in H_0^1(M; E)$ such that
    \begin{equation}
        v_0 = u + \varphi
    \end{equation}
    and $\mathcal{N}_A \varphi|_{\Gamma} = 0$. Thus $\varphi$ admits an extension to $H^1(\widetilde{M}; E)$ by zero, denoted by the same letter. We introduce
    \begin{equation}
        \widetilde{v} := \widetilde{u} + \varphi
    \end{equation}
    so that $\widetilde{v} \in H^1(\widetilde{M}; \widetilde{E})$, $\widetilde{v} = v_0$ in $M$ and $\widetilde{v} = \widetilde{u}$ in $\widetilde{M} \setminus M$. Now for $w \in H^1(\widetilde{M}; \widetilde{E})$, we have (the subscript denotes the domain where we take the $L^2$ inner product)
    \begin{align}
        (\Lapl_{\widetilde{B}}\widetilde{v}, w)_{\widetilde{M}} &= (\Lapl_B v_0, w)_M + (\Lapl_{\widetilde{B}} \widetilde{u}, w)_{\widetilde{M} \setminus M}\\
        &= (\Lapl_{\widetilde{A}} \widetilde{w}, w)_{\widetilde{M} \setminus M} = 0
    \end{align}
    Here we used that $\Lapl_B v_0 = 0$ and $\widetilde{A} = \widetilde{B}$ in $\widetilde{M} \setminus M$. Therefore, $\Lapl_{\widetilde{B}} \widetilde{v} = 0$.
    
    We also have 
    \begin{align}
        \widetilde{v}|_{\partial \widetilde{M}} &= (\widetilde{u} + \varphi)_{\partial \widetilde{M}} = \widetilde{u}|_{\partial \widetilde{M}}\\
        \mathcal{N}_{\widetilde{A}} \widetilde{u}|_{\widetilde{\Gamma}} &= \mathcal{N}_{\widetilde{B}} \widetilde{v}|_{\widetilde{\Gamma}}
    \end{align}
    by the properties of $\varphi$ and $\widetilde{\Gamma}$. Therefore, the Cauchy data sets of $\Lapl_{\widetilde{A}}$ and $\Lapl_{\widetilde{B}}$ agree on $\widetilde{\Gamma}$ and since these are graphs of the DN maps, we finish the proof.
\end{proof}

With these two Lemmas behind our back, we are ready to prove the main theorem of the section.

\begin{proof}[Proof of Theorem \ref{thm:mainthm'}]
    Assume first $n \geq 3$. Let $p \in \Gamma$. By Lemma \ref{lemma:extension}, we may extend $M$ near $p$ along with the bundles $E$ and $E'$ to $\widetilde{E}$ and $\widetilde{E'}$ respectively, such that if we denote the ball we added to $M$ to form $\widetilde{M}$ by $D$, we have $\widetilde{E}|_{D} = \widetilde{E'}|_D$. Furthermore, we extend connections (after a gauge transform) such that $\widetilde{A} = \widetilde{B}$ on $D$ and the Hermitian structures also agree on $D$. Then $\Lambda_{\widetilde{A}}(f)|_{\widetilde{\Gamma}} = \Lambda_{\widetilde{B}}(f)|_{\widetilde{\Gamma}}$ for all $f \in C_0^\infty(\widetilde{\Gamma}; \widetilde{E})$ by the same Lemma. 
    
    We want to prove that the two connections $A$ and $B$ have the same holonomy at $p$, which would yield they are equivalent; we construct the isomorphism between $E$ and $F$ along the way. We fix an embedded smooth closed curve $\gamma$, centred at $p$. Fix a small $\varepsilon > 0$ such that the $\varepsilon$ neighbourhood of $p$ in $\partial M$, denoted by $U_{\varepsilon}$ is contained in $D \cap \Gamma$. The curve $\gamma$ is approximated in $C^1$ norm by the embedded curves $\gamma_\varepsilon: [0, 2] \to M$, such that $\gamma_{\varepsilon}(0) = \gamma_{\varepsilon}(2) = p$, $\gamma_{\varepsilon}(1) = p_1 \in U_\varepsilon$, $\gamma_{\varepsilon}(0, 1) \subset M^\circ$ and $\gamma_{\varepsilon}[1, 2] \subset U_{\varepsilon}$. Thus it suffices to prove the holonomies of $A$ and $B$ along $\gamma_\varepsilon$ agree. We relabel $\gamma_\varepsilon$ by $\gamma$ and denote $\gamma_1 := \gamma|_{[0, 1]}$, $\gamma_2 := \gamma|_{[1, 2]}$.
    
    For $\gamma_2$ it is easy to see that the holonomies agree, as $A = B$ on $\partial M$ by assumption.
    
    For $\gamma_1$, we need a density argument of Runge-type. By Lemma \ref{lemma:runge}, there are smooth twisted-harmonic with respect to $\widetilde{A}$ sections $s_i \in C^\infty(\overline{\widetilde{M}}; \widetilde{E})$ for $i = 1, \dotso, m$, i.e. $\Lapl_{\widetilde{A}} s_i = 0$ and we also have $\supp{(s_i|_{\widetilde{M}})} \subset \Gamma$. Furthermore, we also ask that
    \begin{equation}\label{eq:spann}
        \spann\{s_1(x), \dotso, s_m(x)\} = E(x)
    \end{equation}
    for $x \in \gamma_1$.
    
    Then we construct twisted-harmonic sections with respect to $\widetilde{B}$ by solving the Dirichlet problem, such that 
    \begin{align}
        \Lapl_{\widetilde{B}} r_i &= 0\\
        r_i|_{\partial \widetilde{M}} &= s_i|_{\partial \widetilde{M}}
    \end{align}
    Therefore, we have that $s_i = r_i$ on $D$, since by the assumption on DN maps and $\widetilde{A} = \widetilde{B}$ on $D$
    \begin{equation}
        \mathcal{N}_{\widetilde{A}} s_i = \mathcal{N}_{\widetilde{A}} r_i
    \end{equation}
    So by unique continuation, we have the claim.
    
    Consider $\delta > 0$ small enough such that the $\delta$-tubular neighbourhood $T$ around $\gamma_1$ intersects $\partial M$ inside $U_\varepsilon$. Then $E|_T = E'|_T = T \times \mathbb{C}^m$ since $T$ contractible. We introduce the $m \times m$ matrices of sections over $T$:
    \begin{align}
        F = \big(s_1(x), s_2(x), \dotso, s_m(x)\big) \quad \text{ and } \quad G = \big(r_1(x), r_2(x), \dotso, r_m(x)\big)
    \end{align}
    Then we have $\Lapl_{\widetilde{A}} F = 0$ and $\Lapl_{\widetilde{B}} G = 0$ by extending the action diagonally. Take $\delta$ small enough such that over $T$ we have $F$ non-singular, by \eqref{eq:spann}.
    
    As in the proofs of Theorems \ref{linebundle} and \ref{analytic_m>1}, and also Lemma \ref{gauge_constr}, we introduce the auxiliary connection $A'$ over $T$:
    \begin{equation}
        A' := F^*A
    \end{equation}
    In the vicinity of $p$, we may also analogously introduce $B' = G^*B$ as $G$ is non-singular. Then by the unique continuation principle, since $A'$ and $B'$ satisfy elliptic equations by Lemma \ref{gauge_constr}, we have $B' = A'$ in a neighbourhood $V$ of $p$ inside $T$. In other words, after introducing $H = FG^{-1}$, we have
    \begin{equation}
        A' = B' \iff H^*A = B
    \end{equation}
    in $V$. But since $A$ and $B$ are Hermitian, this implies that $H$ is unitary on $V$. 
    
    Assume there is a point $q \in T \cap \partial V$ such that there is a sequence of points $q_k \in V$ with $q_k \to q$ and $\det G(q_k) \to 0 = \det G(q)$ as $n \to \infty$. Then as $H$ is unitary, we have as $n \to \infty$
    \begin{equation}\label{eq:unitary}
        |\det G (q_k)| = |\det F (q_k)| \to 0
    \end{equation}
    But this contradicts the fact that $\det F \neq 0$ on the compact set $T$. Thus $\det G \neq 0$ on $T$ and so the connection $B' = G^*B$ is defined over the set $T$, so by unique continuation $H^*A = B$ in $T$. Therefore, since $F = G$ on $U_\varepsilon$, $H = Id$ on the same set and $A$ and $B$ have the same holonomy.
    
    Note that the same argument gives that the parallel transport matrices for $A$ and $B$ are equal over any embedded curve starting at a point $p \in \Gamma$ and endpoint $p_1 \in \Gamma$ (we extend the manifold and the bundles by gluing two balls in this case). We may apply this if $\Gamma$ has several connected components.
    
    Now the usual argument, given in the proof of Theorem 7.3. \cite{cek1} gives us that the parallel transport $H$ of $Id$ from $p$ in the bundle of homomorphisms $\text{Hom}(E', E)$ with the auxiliary connection $\nabla^{\text{Hom}}$, given by $\nabla^{\text{Hom}} u := \nabla_A u - u\nabla_B$ for $u \in C^\infty\big(M; \text{Hom}(E', E)\big)$, is independent of the path and gives us the desired automorphism. It is clear that $H$ is non-singular and unitary from the first order PDE it satisfies.
    
    For $n = 2$, the loop $\gamma$ at $p$ could have self-intersections, due to codimension reasons. To treat this case, we first note that the discussion above carries over to the case of loops without self-intersections, i.e. we have $\mathcal{P}^A = \mathcal{P}^B$ along embedded loops (here $\mathcal{P}$ denotes parallel transport). Then Lemma \ref{lemma:geomlemma} applies to identify holonomies of $A$ and $B$ and we conclude the proof in the same way as for $n \geq 3$.
    
\end{proof}

We now prove a simple geometric lemma, that we used in the proof of the previous theorem in the case of surfaces.

\begin{lemma}\label{lemma:geomlemma}
    Let $\Sigma$ be a smooth compact surface with boundary and $E$ a vector bundle over $\Sigma$ equipped with two connections $A$ and $B$. Let $p \in \Sigma$ and assume $\mathcal{P}^A = \mathcal{P}^B$ for all embedded simple closed curved at $p$ (here by $\mathcal{P}$ we denote parallel transport). Then $\mathcal{P}^A = \mathcal{P}^B$ along any loop, i.e. $A$ and $B$ have the same holonomy at $p$.
\end{lemma}
\begin{proof}
    Let $\gamma: [0, 1] \to M$ be a loop at $p$. By differential topology \cite{hirsch} we may assume $\gamma$ is embedded except at most finitely many points, where it self-intersects transversely. Let the number of self-intersections be $k$. We inductively prove that $\mathcal{P}^A = \mathcal{P}^B$, by induction on $k$. The case $k = 0$ holds true by assumption.
    
    We now prove the inductive step and assume the claim holds for up to $k$ self-intersections. Let $\gamma(t_0) = \gamma(t_1) = q$ with $t_1 > t_0 > 0$ minimal, i.e. the first intersection. Denote $\gamma_1 := \gamma|_{[0, t_0]}, \gamma_2:= \gamma|_{[t_0, t_1]}$ and $\gamma_3 := \gamma|_{[t_1, 1]}$ and corresponding parallel transports by $\mathcal{P}_1, \mathcal{P}_2$ and $\mathcal{P}_3$. By the inductive hypothesis, we have:
    \begin{align}
        \mathcal{P}_1^A \cdot \mathcal{P}_3^A &= \mathcal{P}_1^B \cdot \mathcal{P}_3^B\\
        \mathcal{P}_1^A \cdot (\mathcal{P}_2^A)^{-1} \cdot \mathcal{P}_3^A &= \mathcal{P}_1^B \cdot (\mathcal{P}_2^B)^{-1} \cdot \mathcal{P}_3^B
    \end{align}
    But by substituting the first equation above into the second (on both sides) and inverting, we get:
    \begin{equation}
        \mathcal{P}^A = \mathcal{P}_1^A\cdot \mathcal{P}_2^A \cdot \mathcal{P}_3^A = \mathcal{P}_1^B\cdot \mathcal{P}_2^B \cdot \mathcal{P}_3^B = \mathcal{P}^B
    \end{equation}
    This concludes the proof of induction and of the lemma.
\end{proof}

\subsection{The non-unitary case.}
Finally, we prove a version of the previous theorem in the more general case when the connections are non-unitary. Note that crucially, when either the gauge $F$ from Lemma \ref{lemma} constructed by solving $d_A^*d_A F = 0$ is unitary \emph{or} if the connection is unitary, then $A' = F^*A$ satisfies an elliptic system with coefficients independent of the gauge $F$. This enables us to apply the UCP in the main theorem in the either the case of unitary connections or when we can construct unitary gauges. Also, we can handle connections on line bundles, since the Yang-Mills equation there is locally simply $d^*dA = 0$ and is independent of the Hermitian structure.

In this section, we prove the analogue of Proposition \ref{prop:linebundlenonunitary}.

\begin{prop}\label{prop:rungenonunitary}
    Conclusions of Theorem \ref{thm:mainthm'} hold in the case of general $GL(1, \mathbb{C})$ Yang-Mills connections (i.e. in the $m = 1$ case).
\end{prop}
\begin{proof}
    We follow the proof of Theorem \ref{thm:mainthm'} and only point out the main differences; this is similar in spirit to the proof of Proposition \ref{prop:linebundlenonunitary}.
    
    We take a point $p \in \Gamma$, extend the domain near this point, assume we have curves $\gamma_1$ and $\gamma_2$ as before and we form the tube $T$ around $\gamma_1$, such that the functions $F$ and $G$ over the trivialisation $E|_{T} = T \times \mathbb{C} = E'|_{T}$ satisfy $\Lapl_A F = 0 = \Lapl_B G$ and $F|_T$ is non-singular. Then by the UCP we have $F^*A = A' = B' = G^*B$ inside an open, connected set, containing a neighbourhood of $\{0\} \times B$, where we identify $T = [0, 1] \times B$ with $B$ the $(n-1)$-dimensional unit ball. Here we crucially use that $m = 1$, so that $A'$ and $B'$ satisfy the same equation.
    
    We want to prove that $G|_T$ is non-singular too, but we cannot use \eqref{eq:unitary}. Note that on $V$
    \begin{equation}\label{eq:auxm=1nonunitary}
        H^*A = H^{-1}dH + A = B
    \end{equation}
    This implies the following two easy facts, again on $V$:
    \begin{equation}\label{eq:aux}
        dA - dB = 0 \quad \text{ and } \quad dH = H \cdot (B - A)
    \end{equation}
    The first fact follows from taking traces of $H^*F_A = F_B$ ($F_A$ and $F_B$ denote the curvature) and the second from \eqref{eq:auxm=1nonunitary}. Consider
    \begin{equation}
        s := \sup \{t \in [0, 1] \mid G \neq 0 \text{ on } [0, t) \times B\}
    \end{equation}
    Clearly $s > 0$ and we want to prove $s = 1$. Take $U := [0, s) \times B \subset T$ and assume there is a point $q \in \partial U$ with $G (q) = 0$, or equivalently $H(q) = \infty$.
    
    But as $U$ is simply-connected and \eqref{eq:aux} holds, we find a smooth $f$ on $U$ with
    \begin{equation}\label{eq:localansatz}
        df = B - A
    \end{equation}
    Then $H = e^f$ up to constant, so by assumption $f(q) = \infty$. But by \eqref{eq:localansatz} and the Mean value theorem, we get $f$ uniformly bounded on $U$, contradiction.
    
    
    Thus $A' \equiv B'$ on $T$ and the holonomies of $A$ and $B$ along $\gamma$ are equal, which concludes the proof.
\end{proof}

\begin{rem}\rm
    If we proved existence of a unitary gauge $F$ near closed curves in the interior, we would be able to prove the main theorem for non-unitary connections. Note that these gauge exist locally, similar to the fact that Coulomb gauges exist with value in any compact Lie group (see \cite{uhlenbeck}). It might be sufficient to approximate a unitary gauge constructed locally using Runge approximation theorem, but we did not pursue this approach here.
\end{rem}

\appendix

\section{The space of smooth curves and an extension lemma}\label{app:A}
We need the metric space of smooth curves in the proof of our main theorem -- here are some properties:

\begin{rem}\rm \label{pathmetric}
We are using the standard metric on the space $C^{\infty}([0, 1]; \mathbb{R})$ induced by the seminorms $\lVert{f}\rVert_k = \sup_{t \in [0, 1]}{\big|\frac{d^k f}{d t^k}\big|}$. Then a choice of the metric on this space is:
\begin{equation*}
d(f, g) = \sum_{k = 0}^{\infty} 2^{-k} \frac{\lVert{f - g}\rVert_k}{1 + \lVert{f - g}\rVert_k}
\end{equation*}
and it is a standard fact that this space is a Fr\'{e}chet space with the same topology as the weak topology given by the seminorms. Furthermore, this also induces a Fr\'{e}chet metric to the space $C^{\infty}([0, 1]; \mathbb{R}^m) = \oplus_{i = 1}^m C^{\infty}([0, 1]; \mathbb{R})$ for all $m \in \mathbb{N}$. Moreover, we may consider the space $C^{\infty}([0, 1]; M)$ for any compact Riemannian manifold $(M, g)$ by isometrically embedding $M$ into a Euclidean space $\mathbb{R}^N$ for some $N$, as a closed subspace of $C^{\infty}([0, 1], \mathbb{R}^N)$.
\end{rem}

Now we prove the following lemma for the continuity of $h$ in the interior and on the boundary of the manifold. 

\begin{lemma}\label{extension_lemma}
Let $\Omega \subset \mathbb{R}^n$ be a domain and $E \subset \Omega$ a closed subset. Assume also that for any two points $x, y \in \Omega \setminus E$ and any smooth path $\gamma$ in $\Omega$ between $x$ and $y$, there exist smooth paths $\gamma_i$ from $x$ to $y$, lying in $\Omega \setminus E$, for $i = 1, 2, \dotso$, that converge to $\gamma$ in the metric space $C^{\infty}([0, 1]; \mathbb{R}^n)$. Let $f: \Omega \setminus E \to \mathbb{C}$ be a smooth function, such that $\partial^\alpha f$ extend continuously to $\Omega$ for all multi-indices $\alpha$. Then there exists a unique smooth extension $\tilde{f}:\Omega \to \mathbb{C}$ with $\tilde{f}|_{\Omega \setminus E} = f$.
\end{lemma}
\begin{proof}
This is a local claim, so we will consider an extension near a point $x \in E$. We will prove that the continuous extension $\tilde{f}$ of $f$ to $\Omega$ is differentiable with the derivative given by the continuous extension $h$ of $df$ to $\Omega$. By inductively repeating the argument for all $\partial^\alpha f$ for multi-indices $\alpha$, it clearly suffices to prove this.

Consider the point $y = x + \delta e_1$, where $\delta >0$ is small enough so that the straight line path $\gamma$ between $x$ and $y$ lies in $\Omega$. Since $\Omega \setminus E$ is dense in $\Omega$, we may choose points $x', y' \in \Omega \setminus E$ that are close to $x, y$, respectively. Consider the path $\gamma'$ obtained by smoothing out the straight line path from $x'$ to $x$, $\gamma$ and the straight line path from $y$ to $y'$. By the hypothesis, there exists a sequence of paths $\gamma_n$ with endpoints at $x'$ and $y'$, lying entirely in $\Omega \setminus E$ that converge to $\gamma'$ in the path metric.

We will consider the integrals along the curves $\gamma_n$: after possibly reparametrising, we may assume that $\gamma_n$ are parametrised by arc-length -- we can always do this for $n$ sufficiently large, as $\gamma$ has a nowhere zero derivative. Therefore, we may integrate $h(\dot{\gamma}_n)$ to get that, by the Fundamental Theorem of Calculus
\begin{equation*}
f(y') - f(x') =  \int_{\gamma_n}d(f \circ \gamma_n(t))= \int_{\gamma_n}h(\dot{\gamma}_n)
\end{equation*}
Here, we think of $h$ as given by the vector of partial derivatives of $f$. By uniform convergence of the curves, we immediately get that
\begin{equation*}
f(y') - f(x') = \int_{\gamma_n}h(\dot{\gamma}_n) \to \int_{\gamma'}h(\dot{\gamma'})
\end{equation*}
and therefore, if we take $x' \to x$ and $y' \to y$ (we can do this as $\Omega \setminus E$ is dense in $\Omega$), we get:
\begin{equation*}
\frac{\tilde{f}(x + \delta e_1) - \tilde{f}(x)}{\delta} = \frac{1}{\delta}\int_0^\delta h_{x + te_1}(e_1)dt \to h_x(e_1)
\end{equation*}
as $\delta \to  0$. Therefore, the partial derivative in the $e_1$ direction exists and similarly, all other partials exist and are equal to the components of $h$. This finishes the proof.
\end{proof}

\begin{rem}\rm\label{rmk_boundary_extension}
If we are given a smooth function $f$ in the interior of a domain $\Omega \subset \mathbb{R}^n$ with smooth boundary, such that all derivatives $\partial^\alpha f$ extend continuously to the boundary, it is well known that there exists a smooth extension $\tilde{f}$ to $\mathbb{R}^n$, such that $\tilde{f}|_{\Omega} = f$. This remark, together with the above lemma, are used in the proof of the smooth extension of $h$ over the singular set in Theorem \ref{linebundle}.
\end{rem}


Finally, we would like to recall the well-posedness conditions under which the solution operator to a generalised heat equation is smoothing. One set of such conditions is given by (1.5)-(1.7) on page 134 in Treves \cite{treves} -- we state them here for completeness. Let $X$ be a manifold of dimension $n$ and $t$ a variable in the real line $\mathbb{R}$; we will consider vector functions with values in the finite dimensional space $H = \mathbb{C}^m$. Let $A(t)$ be a pseudodifferential operator of order $k$ with values in $L(H) = \mathbb{C}^{m \times m}$ depending smoothly on $t \in [0, T)$; this means that in a local chart $\Omega \subset X$ we have the symbol of $A(t)$ modulo $S^{-\infty}$ being a smooth function $a_\Omega(x, t, \xi): [0, T) \to S^k(\Omega; L(H))$. We consider the following equation in $X \times [0, T)$, where $U$ valued in $H$:
\begin{align*}
\frac{dU}{dt} - A(t) \circ U \equiv 0 \quad \text{modulo} \quad S^{-\infty}
\end{align*}
The set of conditions for this equation to be well-posed is the following:
\begin{condition}[Well-posedness of the heat equation]\label{condtreves}
For every local chart $\Omega \subset X$, there is a symbol $a(x, t, \xi)$ depending smoothly on $t \in [0, T)$ and defining a pseudodifferential operator $A_\Omega(t)$ congruent to $A(t)$ modulo regularising operators in $\Omega$, such that for every compact $K \subset \Omega \times [0, T)$ there is a compact subset $K'$ of the open half-plane $\mathbb{C}_- = \{z \in \mathbb{C} \mid \re(z) < 0\}$ such that
\begin{align}\label{condtreveseqn}
z \times Id - \frac{a(x, t, \xi)}{(1 + |\xi|^2)^{\frac{m}{2}}}: H \to H 
\end{align}
is a bijection for all $(x, t) \in K$, $\xi \in \mathbb{R}^n$ and $z \in \mathbb{C} \setminus K'$.
\end{condition}

One remark is in place after this condition:

\begin{rem}\rm
In fact, the symbol of the Laplace operator in the ordinary heat equation does not immediately satisfy Condition \ref{condtreves} for a well-posed (generalised) heat equation  -- if one plugs $-|\xi|^2$ ($m = 1$) into $\eqref{condtreveseqn}$, we have that the zero set spreads such that we have $\re(z) \in (-1, 0)$ and $\im(z) = 0$, which is certainly not contained in a compact subset of $\mathbb{C}_{-} = \{\re(z) < 0\}$; the trick is to add a factor of $e^{-|\xi|^2}$ which does not change the class of the symbol modulo $S^{-\infty}$, as we will see in the proof of the Lemma below.
\end{rem}

Using the idea in the above remark, we prove that the operator we use in Proposition \ref{PDO} satisfies Condition \ref{condtreves}:

\begin{lemma}\label{WPHEQN}
The $\mathbb{C}^{m \times m}$-valued pseudodifferential operator $A = B - E \times Id$ (defined in Lemma \ref{factorisationlemma}) satisfies Condition \ref{condtreves}.
\end{lemma}
\begin{proof}
Denote by $a_1 = -\sqrt{q_2} = -\sqrt{\sum_{\alpha, \beta} g^{\alpha\beta}\xi_\alpha \xi_\beta}$ the principal symbol of $A$ ($E$ has degree zero). If $K \subset [0, T] \times \mathbb{R}^{n-1}$ compact, then there exist positive $C_1$, $C_2$ and $c$ such that
\begin{align*}
c|\xi| \leq |a_1(x, t, \xi)| &\leq C_1(1 + |\xi|^2)^{\frac{1}{2}}\\
|a_0 (x, t, \xi)| &\leq C_2
\end{align*}
for all $(x, t) \in K$ and $\xi \in \mathbb{R}^{n-1}$, by definition of symbols and the fact that $g^{\alpha \beta}$ is positive definite. Thus we can rewrite:
\begin{align}\label{expression1}
z \times Id - \frac{-\sqrt{q_2}\times Id + a_0}{(1 + |\xi|^2)^{\frac{1}{2}}} = \Big(z + \frac{\sqrt{q_2}}{(1 + |\xi|^2)^{\frac{1}{2}}}\Big) \times Id - \frac{a_0}{(1 + |\xi|^2)^{\frac{1}{2}}}
\end{align}
and if this expression is singular, we ought to have 
\begin{align}\label{ineq1}
\frac{|a_0|^2}{1 + |\xi|^2} \geq m^2 \Big| z + \frac{\sqrt{q_2}}{(1 + |\xi|^2)^{\frac{1}{2}}}\Big|^2 = m^2 |s|^2 + m^2\Big(r + \frac{\sqrt{q_2}}{(1 + |\xi|^2)^{\frac{1}{2}}}\Big)^2
\end{align}
where $z = r + is$. If we had $|\xi|$ large enough and $r \geq -\epsilon$ for some small $\epsilon > 0$, the left hand side of \eqref{ineq1} would be small and the right hand side of it would be bigger than $s^2 + (r + \frac{c}{2})^2$ (up to a constant). Therefore for $|\xi| \geq K$ for some $K$, \eqref{expression1} will be non-singular for $r \geq -\epsilon$.

Notice that in the condition we have the freedom of adding a smoothing factor -- this will take care of the singular behaviour for $|\xi|$ in a compact set. We will add a factor of $Ce^{-|\xi|^2} \times Id \in S^{-\infty}$ for some $C>0$ to remedy this. First of all, notice that the above argument remains the same with the same $|\xi|$, if we consider the symbol $\sqrt{q_2} \times Id + a_0 + Ce^{-|\xi|^2} \times Id$.

Furthermore, we have the left hand side of \eqref{ineq1} bounded for all $\xi$ uniformly, whereas the right hand side is bigger (up to a constant) than $(Ce^{-|\xi|^2} - \epsilon)^2$ for $r \geq -\epsilon$, large enough $C$ and $|\xi| \leq K$. Clearly this inequality fails to hold for large $C$ and this finishes the proof.
\end{proof}

\section{Runge approximation}\label{app:B}

In this section, we give an argument that approximates a given function on an embedded curve in the interior of a compact manifold with boundary, by solutions to an elliptic equation which are compactly supported at a prescribed open set at the boundary. The results can be easily generalised to arbitrary elliptic operators with diagonal principal part and smooth coefficients.

For this, we will need a unique continuation result and some well-posedness for elliptic boundary value problems in negative Sobolev spaces. We recall some definitions first.

Let $(M, g)$ be a compact smooth Riemannian manifold with boundary, $E$ a smooth Hermitian vector bundle equipped with a smooth unitary connection $A$ and let $\Gamma \subset \partial M$ be an open set. We denote the twisted Laplacian acting on sections $C^\infty(M; E)$ by $\Lapl = d_A^*d_A$. We recall that by the usual theory (see the appendix in \cite{cek3} for details) that the problem:
\begin{align*}
    \Lapl u = 0, \quad u|_{\partial M} = f
\end{align*}
is uniquely solvable for $f \in H^{\frac{1}{2}}(\partial M; E)$ and yields a solution $u \in H^1(M; E)$. Then the covariant normal derivative is defined weakly as $d_A(u)(\nu) \in H^{-\frac{1}{2}}(\partial M; E)$, by its action on $H^{\frac{1}{2}}(M; E)$.

We proceed to define the Sobolev spaces, for $s \geq 0$:
\begin{align*}
    \widetilde{H}^s(\Gamma; E) := \text{closure of } \{g \in H^s(\partial M; E) \mid \supp{g} \subset \Gamma\} \text{ in } H^s(\partial M; E)
\end{align*}
Then the dual of this space is given by $\big(\widetilde{H}^s(\Gamma; E)\big)' = H^{-s}(\Gamma; E)$ (see \cite{RS17} and references therein).

Now assume $M_1$ is a smooth submanifold (zero codimension) with smooth boundary, compactly contained in $M^\circ$. We define the following spaces of solutions, for $s \geq \frac{3}{2}$:

\begin{align}
    S_1 &:= \{u \in H^{s - \frac{1}{2}}(M_1; E) \mid \Lapl u = 0\} \subset H^{s - \frac{1}{2}}(M_1; E)\\
    X &:= \text{closure of $S_1$ in } H^{s - \frac{1}{2}}(M_1; E)\\
    S &:=  \{u \in H^{s + \frac{1}{2}}(M; E) \mid \Lapl u = 0, \, u|_{\partial M} \in \widetilde{H}^s(\Gamma; E)\}
\end{align}

Moreover, we define the restriction map for $s \geq \frac{3}{2}$:
\begin{align}
    A: \widetilde{H}^s(\Gamma; E) &\to X \subset H^{s - \frac{1}{2}}(M_1; E)\\
    g &\mapsto u|_{M_1}
\end{align}
where $u \in H^{s + \frac{1}{2}}(M; E)$ is the unique solution to $\Lapl u = 0$ with $u|_{\partial M} = g$. Then the dual map $A'$ to $A$ satisfies:
\begin{align}\label{eq:dualmap}
    A': \big(H^{s - \frac{1}{2}}(M_1; E)\big)' &\to H^{-s}(\Gamma; E)\\
    h &\mapsto d_A(w)(\nu)|_{\Gamma}
\end{align}
where the prime denotes the topological dual and $w \in H^{-s + \frac{3}{2}}(M; E)$ is given by solving the Dirichlet problem
\begin{align}\label{eq:dirichletnegsob}
    \Lapl^* w = \begin{cases}
                    h, \text{ in } M,\\
                    0 \text{ in } M\setminus M_1
                \end{cases}
\end{align}
with $w = 0$ on $\partial M$. Note that $\big(H^{s - \frac{1}{2}}(M_1; E)\big)' \subset H^{-(s - \frac{1}{2})}(M_1; E) = \big(H_0^{s - \frac{1}{2}}(M_1; E)\big)'$. Here we note that the well-posedness theory of \eqref{eq:dirichletnegsob} is not trivial and follows from Proposition \ref{prop:negsobolev}, by noting that the right hand side of \eqref{eq:dirichletnegsob} is compactly supported in the interior of $M$ and so lies in the allowed space of inhomogeneities $\Xi^{-(s - \frac{1}{2})}(M; E)$ defined around the lines of \eqref{eq:inhomogeneitynegsob}. These technicalities are postponed for later to simplify exposition.

The mapping property \eqref{eq:dualmap} follows from
\begin{equation}\label{eq:widentity}
    (Ag, h)_{L^2(M_1; E)} = (u, \Lapl^*w)_{L^2(M; E)} = \big(g, d_A(w)(\nu)\big)_{L^2(\partial M; E)}
\end{equation}
where the last equality follows from Stokes' theorem. We are ready to make a statement, with notation as above; the proof mimics the proof of Theorem 1 in \cite{RS17}, but it was well-known before -- see e.g. the work of Browder \cite{BII}, Theorem 3.22. The idea is to reduce the statement by duality to a unique continuation principle.

We first prove the theorem, granted technical ingredients proven in Proposition \ref{prop:negsobolev}.

\begin{theorem}\label{thm:rungethm}
    Assume $M \setminus M_1$ is connected and $s \geq \frac{3}{2}$. For any $\varepsilon > 0$, $h \in S_1$, there exists $u \in S$ such that
    \begin{equation}
        \lVert{h - u|_{M_1}}\rVert_{H^{s - \frac{1}{2}}(M_1; E)} < \varepsilon
    \end{equation}
\end{theorem}
\begin{proof}
    It suffices to prove that the range of $A$ is dense in $X$. So by Hahn-Banach, it suffices to prove that for any linear functional $T$ on $H^{s - \frac{1}{2}}(M_1; E)$ with
    \begin{equation}
        T(Ag) = 0
    \end{equation}
    for all $g \in \widetilde{H}^s(\Gamma; E)$, then $Tv = 0$ for all $v \in S$. By duality, every such $T$ is given by an $h \in H^{-(s - \frac{1}{2})}(M_1; E)$ via $T(\cdot) = (\cdot, h)_{L^2}$. 
    
    So, assume $(Ag, h)_{L^2} = 0$ for all $g \in \widetilde{H}^s(\Gamma; E)$; we want to prove $(v, h) = 0$ for all $v \in S$. If $w \in H^{-s + \frac{3}{2}}(M; E)$ satisfies equation \eqref{eq:dirichletnegsob}, then by \eqref{eq:widentity} we get
    \begin{equation}
        \big(g, d_A(w)(\nu)\big)_{L^2(\partial M; E)} = 0
    \end{equation}
    for all $g \in \widetilde{H}^s(\Gamma; E)$. Thus $d_A(w)(\nu)|_{\Gamma} = 0$.
    
    Therefore, $w$ solves:
    \begin{equation}
        \begin{cases}
            \Lapl^* w= 0\\
            w|_{\partial M} = 0\\
            d_A(w)(\nu)|_{\Gamma} = 0
        \end{cases}
    \end{equation}
    Then, by the regularity properties given in Proposition \ref{prop:negsobolev} (b), we have that $w \in C^\infty$ in a neighbourhood of a slightly smaller domain of $\Gamma$; by the UCP for local data for $\Lapl$ we get that $w = 0$ in the same domain. Moreover, we get $w \equiv 0$ on the whole of $M \setminus M_1$. Therefore, $h = \Lapl^*(w|_{M_1})$ and $w|_{\partial M_1} = d_A(w)(\nu)|_{\partial M_1} = 0$.
    
    Finally, if $v \in S_1$ then:
    \begin{equation}
        (v, h)_{L^2(M_1; E)} = \big(v, \Lapl^*(w|_{M_1})\big)_{L^2(M_1; E)} = (\Lapl v, w|_{M_1})_{L^2(M_1; E)} = 0
    \end{equation}
    which finishes the proof.
\end{proof}

We record a simple corollary to this Theorem; this result is similar to Theorem 3.22. and 3.23. \cite{BII}, but the proof is different. The interested reader should consult also the other work of Browder.

\begin{corollary}\label{cor:rungecor}
    Let $\varepsilon > 0$. Then for every $v \in S_1 \cap C^\infty(\overline{M_1}; E)$, there exists $u \in C^\infty(\overline{M}; E)$ with $\supp{(u|_{\partial M})} \subset \Gamma$ such that
    \begin{equation}
        \lVert{v - u|_{M_1}}\rVert_{C^0(M_1; E)} < \varepsilon
    \end{equation}
\end{corollary}
\begin{proof}
    By the Sobolev embedding theorem, $W^{k, p}(M; E)$\footnote{Here we denote by $W^{k, p}(M; E)$ the $L^p$-based Sobolev space with $k \in \mathbb{R}$ derivatives.} continuously embeds into $C^0(\overline{M}; E)$ for Sobolev indices satisfying $kp > n$.
    
    Therefore if we take $s > \max{\{\frac{n+1}{2}, \frac{3}{2}}\}$, then by Theorem \ref{thm:rungethm} and the fact that $C_0^\infty(\Gamma) \subset \widetilde{H}^{s}(\Gamma; E)$ we prove the claim.
\end{proof}

We gather all the results we have used in this section about well-posedness in negative Sobolev spaces and unique continuation in one Proposition. There are similarities with Lemma B.2. \cite{conformalcalderon}.

We first introduce the spaces for which we can allow inhomogeneity. Following \cite{LM72}, Chapter 2, for each $r \geq 0$ we introduce
\begin{equation}
    \mathcal{D}^{-r}(M; E) = \{ u \mid u \in H^{-r}(M; E), \, \Lapl u \in \Xi^{-2 -r}(M; E)\}
\end{equation}
and equip it with the graph norm; it is a Hilbert space. Then by Theorem 6.4. \cite{LM72}, $C^\infty(\overline{M}; E)$ is dense in $\mathcal{D}^{-r}(M; E)$ for all $r \geq 0$ with $r -\frac{1}{2} \not\in\mathbb{Z}$ and $\mathcal{D}^{-r}(M; E) \subset H^{-r}(M; E)$ continuously by definition.

We are left to define the $\Xi$ spaces -- for natural numbers $s$, these are locally modelled on $\Omega \subset \mathbb{R}^n$ with smooth boundary as:
\begin{equation}\label{eq:inhomogeneitynegsob}
    \Xi^s(\Omega; \mathbb{C}^m) := \{u \in L^2(\Omega; \mathbb{C}^m) \mid \rho^{|\alpha|} D^\alpha u \in L^2(\Omega; \mathbb{C}^m), \, |\alpha| \leq s\}
\end{equation}
where $\rho$ is a smooth boundary defining function (positive in the interior, vanishing at $\partial \Omega$) and we set the corresponding norm of $u$ to be the sum of $L^2$ norms of $\rho^{|\alpha|}D^\alpha u$, giving a Hilbert space. For positive real $s$, these spaces are defined by interpolation (see \cite{LM72} for details) and for $s > 0$, we define the negative ones as $\Xi^{-s}(\Omega; \mathbb{C}^m) = (\Xi^s(\Omega; \mathbb{C}^m))'$.

The generalisations of these spaces to manifolds are given in the usual manner.

We also record Theorem 6.5. \cite{LM72} about traces. It says that the maps $T_1:u \mapsto u|_{\partial M}$ and $T_2:u \mapsto d_A(u)(\nu)|_{\partial M}$ extend continuously from $C^\infty(\overline{M}; E)$ to maps
\begin{align}\label{eq:traces}
    T_1: \mathcal{D}^{-r}(M; E) \to H^{-r - \frac{1}{2}}(M; E) \,\,\, \text{ and } \,\,\, T_2: \mathcal{D}^{-r}(M; E) \to H^{-r - \frac{3}{2}}(M; E)
\end{align}

\begin{prop}\label{prop:negsobolev}
    (a) For $s < 0$, the Dirichlet problem
    \begin{align}
        \Lapl u = f \in \Xi^{s - 2}(M; E)\\
        u|_{\partial M} = g \in H^{s - \frac{1}{2}}(M; E)
    \end{align}
    has a unique solution $u \in \mathcal{D}^{s}(M; E)$, where the restriction $u|_{\partial M}$ is interpreted in the sense of equation \eqref{eq:traces}.\\
    (b) If $g \in H^{s - \frac{1}{2}}(M; E)$ is $C^\infty$ near a point $p \in \partial M$ and $f = 0$ near $p$, then $u$ is also $C^\infty$ near $p$.\\
\end{prop}
\begin{proof}
    Part (a) follows from Theorem 6.7. from \cite{LM72}. Part (b) follows from the proof of Lemma B.2. (b) \cite{conformalcalderon} generalised to systems.
\end{proof}


\begin{thebibliography}{aa}

\bibitem{ab}
M. F. Atiyah, R. Bott, \emph{The Yang-Mills equations over Riemann surfaces}, Philos. Trans. Roy. Soc. London Ser. A {\bf 308} (1983), no. 1505, 523--615.

\bibitem{BM13} G. Bal, F. Monard, \emph{Inverse anisotropic conductivity from power densities in dimension $n \geq 3$}, Comm. Partial Differential Equations {\bf 38} (2013), no. 7, 1183--1207. 

\bibitem{bar} 
C. B{\"a}r, \emph{Zero sets of solutions to semilinear elliptic systems of first order}, Invent. Math. {\bf 138} (1999), 183--202.

\bibitem{BII}
F. E. Browder, \emph{Functional analysis and partial differential equations. II}, Math. Ann. {\bf 145} (1961/1962) 810--226.

\bibitem{cek2}
M. Ceki\'c, \emph{Harmonic determinants and unique continuation}, preprint (2018), \url{https://arxiv.org/abs/1803.09182}.

\bibitem{cek1}
M. Ceki\'{c}, \emph{The Calder\'{o}n problem for connections}, Comm. Partial Differential Equations 42 (2017), no. 11, 1781--1836.

\bibitem{cek3}
M. Ceki\'c, \emph{The Calder\'on problem for connections}, PhD thesis, University of Cambridge (2017).

\bibitem{DCH}
L. De Carli, S. M. Hudson, \emph{Geometric remarks on the level curves of harmonic functions} (English summary), Bull. Lond. Math. Soc. {\bf 42} (2010), no. 1, 83--95. 

\bibitem{Dsurvey}
S. K. Donaldson, \emph{Mathematical uses of gauge theory}, in: The Encyclopedia of Mathematical Physics, Elsevier (2006).

\bibitem{donaldson}
S. K. Donaldson, P. B. Kronheimer, The geometry of four-manifolds, \emph{Oxford Mathematical Monographs}, The Clarendon Press, Oxford University Press, New York, 1990.

\bibitem{LCW}
D. Dos Santos Ferreira, C. Kenig, M. Salo, G. Uhlmann, \emph{Limiting Carleman weights and anisotropic inverse problems}, Invent. Math. {\bf 178} (2009), 119--171.

\bibitem{MagU}
D. Dos Santos Ferreira, C. Kenig, J. Sj\"{o}strand, G. Uhlmann, \emph{Determining a magnetic Schr\"{o}dinger operator from partial Cauchy data}, Comm. Math. Phys. {\bf 271} (2007), 467--488.

\bibitem{Eskin}
G. Eskin, \emph{Global uniqueness in the inverse scattering problem for the Schr\"odinger operator with external Yang-Mills potentials}, Comm. Math. Phys. {\bf 222} (2001), 503--531. 

\bibitem{guill}
C. Guillarmou, A. S\'{a} Barreto, \emph{Inverse problems for Einstein manifolds}, Inverse Probl. Imaging {\bf 3} (2009), 1--15. 

\bibitem{hirsch}
M. W. Hirsch, Differential topology, \emph{Graduate Texts in Mathematics} {\bf 33}, Springer-Verlag, New York, 1994.

\bibitem{isakov}
V. Isakov, Inverse problems for partial differential equations, \emph{Applied Mathematical Sciences, 127. Springer-Verlag}, New York (1998). xii+284 pp. 

\bibitem{KN63}
S. Kobayashi, K. Nomizu, Foundations of differential geometry, \emph{Vol I. Interscience Publishers}, a division of John Wiley \& Sons, New York-London (1963) xi+329 pp. 

\bibitem{KOP}
Y. Kurylev, L. Oksanen, G. P. Paternatin, \emph{Inverse problems for the connection Laplacian}, arxiv preprint (2015).

\bibitem{conformalcalderon}
M. Lassas, T. Liimatainen, M. Salo, \emph{The Calder\'{o}n problem for the conformal Laplacian}, arxiv preprint (2016).

\bibitem{calderon_analytic}
M. Lassas, M. Taylor, G. Uhlmann, \emph{The Dirichlet-to-Neumann map for complete Riemannian manifolds with boundary}, Comm. Anal. Geom. {\bf 11} (2003), no. 2, 207--221. 

\bibitem{ML89}
H. B. Jr. Lawson, M-L. Michelsohn, Spin geometry, Princeton Mathematical Series, 38. Princeton University Press, Princeton, NJ (1989). xii+427 pp.

\bibitem{leeuhlmann}
J. M. Lee, G. Uhlmann, \emph{Determining anisotropic real-analytic conductivities by boundary measurements}, Comm. Pure Appl. Math. {\bf 42} (1989), no. 8, 1097--1112.

\bibitem{LM72}
J.L. Lions, E. Magenes, Non-homogeneous boundary value problems
and applications, vol. I. \emph{Springer-Verlag}, 1972.

\bibitem{mazzeo}
R. Mazzeo, \emph{Unique continuation at infinity and embedded eigenvalues for asymptotically hyperbolic manifolds}, Amer. J. Math. {\bf 113} (1991), 25--45.

\bibitem{morrey}
C. B. Morrey Jr., \emph{On the analyticity of the solutions of analytic non-linear elliptic systems of partial differential equations. I. Analyticity in the interior}, Amer. J. Math. {\bf 80} (1958), 198--218.

\bibitem{nakamurauhlmann}
G. Nakamura, G. Uhlmann, \emph{A layer stripping algorithm in elastic impedance tomography}, IMA Vol. Math. Appl. {\bf 90}, Springer, New York (1997), 375--384.

\bibitem{RS17}
A. R\"uland, M. Salo, \emph{Quantitative Runge approximation and inverse problems}, arXiv preprint: \url{https://arxiv.org/abs/1708.06307}.

\bibitem{Shu01}
M. A. Shubin, Pseudodifferential operators and spectral theory (Translated from the 1978 Russian original by Stig I. Andersson), Second edition, Springer-Verlag, Berlin, 2001. xii+288 pp.

\bibitem{treves}
F. Treves, \emph{Introduction to Pseudodifferential Operators and Fourier Integral Operators}, Plenum Press, New York, 1980.

\bibitem{uhlenbeck}
K. K. Uhlenbeck, \emph{Connections with Lp bounds on curvature}, Comm. Math. Phys. {\bf 83} (1982), 31--42. 

\bibitem{survey} G. Uhlmann, \emph{Developments in inverse problems since Calder\'{o}n's foundational paper}, Harmonic
analysis and partial differential equations (Chicago, IL, 1996), 295--345, Chicago Lectures
in Math., Univ. Chicago Press, Chicago, IL (1999).

\bibitem{walker}
R. A. Walker, \emph{Problems in Harmonic Function Theory}, Honors Theses (1998), Paper 492.

\end{thebibliography}


\end{document}